\documentclass[12pt]{amsart}
\usepackage[utf8]{inputenc}
\usepackage{graphicx}
\usepackage{epstopdf}
\usepackage{inputenc}
\usepackage{enumitem}
\usepackage{amsmath}
\usepackage{xcolor}
\usepackage{amssymb}
\usepackage{mathrsfs}

\usepackage[utf8]{inputenc}
\usepackage{graphicx}
\usepackage{epstopdf}
\usepackage{inputenc}
\usepackage{enumitem}
\usepackage{amsmath}
\usepackage{xcolor}
\usepackage{amssymb}
\usepackage{dsfont}
\usepackage[colorlinks]{hyperref}
\usepackage{bm}

\usepackage{cleveref}

\crefformat{section}{\S#2#1#3} 
\crefformat{subsection}{\S#2#1#3}

\def\house#1{\setbox1=\hbox{$\,#1\,$}%
	\dimen1=\ht1 \advance\dimen1 by 2pt \dimen2=\dp1 \advance\dimen2 by 2pt
	\setbox1=\hbox{\vrule height\dimen1 depth\dimen2\box1\vrule}%
	\setbox1=\vbox{\hrule\box1}%
	\advance\dimen1 by .4pt \ht1=\dimen1
	\advance\dimen2 by .4pt \dp1=\dimen2 \box1\relax}

\newtheorem{theorem}{Theorem}
\newtheorem*{theorem*}{Theorem}

\newtheorem{lemma}{Lemma}
\newtheorem{conjecture}{Conjecture}
\newtheorem*{acknowledgements*}{Acknowledgements}
\newtheorem{remark}{Remark}

\newcommand{\ve}{\varepsilon}

\newcommand{\mcal}{\mathcal}
\newcommand{\mbb}{\mathbb}

\newcommand\blfootnote[1]{
	\begingroup
	\renewcommand\thefootnote{}\footnote{#1}
	\endgroup
}

\begin{document}
\title[Large values of quadratic Dirichlet $L$-functions]{Large values of quadratic Dirichlet $L$-functions}
\author{Pranendu Darbar and Gopal Maiti }
\address[Pranendu Darbar]{Max Planck Institute of Mathematics\\ Vivatsgasse 7, 53111 Bonn}
\email{darbar@mpim-bonn.mpg.de}

\address[Gopal Maiti]{Max Planck Institute of Mathematics\\ Vivatsgasse 7, 53111 Bonn}
\email{maiti@mpim-bonn.mpg.de}


 \begin{abstract}
Assuming the Generalized Riemann Hypothesis (GRH), we utilize the long resonator method to derive $\Omega$-results for the family of quadratic Dirichlet $L$-functions $L(\sigma, \chi_d)$, where $d$ runs over all fundamental discriminants with $|d| \leq X$ and $\sigma\in [1/2, 1]$ is fixed. This study advances understanding of the maximum size of $L(\sigma, \chi_d)$ within the segment $\sigma\in [1/2, 1]$. In particular, we improve upon Soundararajan's results at the central point and provide a lower bound on the proportion of fundamental discriminants, uniformly within an expected order of magnitude, up to optimal values of the constant for a fixed $\sigma \in (1/2, 1]$.
 \end{abstract}


\maketitle
  
\section{Introduction}

The understanding of maximum size of $L$-functions poses a profound challenge in number theory. Significant advancements have been made in recent years, particularly regarding the Riemann zeta function and Dirichlet $L$-functions in the conductor aspect. After the work of Balasubramanian and Ramachandra \cite{BR} employing the moment method for the Riemann zeta function, Soundararajan \cite{Sound} applied a more versatile approach called the ``resonance method". The resonance method reframes the problem as an optimization problem of a quadratic form applicable to a broad family of $L$-functions. Notably, for the Riemann zeta function, Soundararajan \cite[Theorem 1]{Sound} established that
\[
\max_{T\leq t\leq 2T}|\zeta(1/2+it)|\geq \exp\left((1+o(1))\sqrt{\frac{\log T}{\log_2 T}}\right).
\]

Using the long resonators introduced in \cite{Ais}, Bondarenko and Seip \cite{BS} achieved significant advancements in $\Omega$-results for the Riemann zeta function on the critical line, where they accomplished this feat by estimating a specific type of greatest common divisor (GCD) sum. The efficacy of their long resonator technique extends beyond the critical line, encompassing the critical strip for both the Riemann zeta function and Dirichlet $L$-functions in the conductor aspect are explored in \cite{AMM, AMMP, BS3, Yang}.
In fact, Bondarenko and Seip \cite[Theorem 1]{BS} proved that
\[
\max_{0\leq t\leq T}|\zeta(1/2+it)|\geq \exp\left((1+o(1))\sqrt{\frac{\log T \log_3 T}{\log_2 T}}\right),
\]
a result that was later refined by de la Bret\`{e}che and Tenenbaum \cite{BT}, who enhanced the exponent by replacing $1$ with $\sqrt{2}$ through the optimization of a full GCD-sum problem. Furthermore, for Dirichlet $L$-functions, de la Bret\`{e}che and Tenenbaum \cite[Theorem 1.5]{BT} showed that 
\begin{align*}\label{large values of l functions}
\max_{\substack{\chi\in X_q^+\\ \chi \neq \chi_0}} |L(1/2, \chi)|\geq \exp\left((1+o(1))\sqrt{\frac{\log q \log_3 q}{\log_2 q}}\right),
\end{align*}
where $X_q^+$ denotes the set of all even characters modulo $q$.

In these families, key characteristics include full orthogonality, positive coefficients of both the resonator and $L$-functions, and a connection to the GCD sum, which constitute essential components of the long resonator technique. Though Soundararajan's resonator technique can be extended to a broader class of $L$-functions and Dirichlet polynomials (refer to \cite{AP, XY} for further details). The maximum sizes of $L$-functions over different families were conjectured by Farmer, Gonek and Hughes \cite{FGH}. 
One particularly intriguing example is the family of quadratic Dirichlet $L$-functions, parameterized by fundamental discriminants. For $\sigma \in [1/2, 1]$, this family is represented as
\[
\{L(\sigma, \chi_d): d \text{ is fundamental discriminant}\}.
\]
These $L$-functions feature coefficients that are not necessarily positive and do not exhibit full orthogonality, making the conventional use of long resonators challenging, as discussed by C. Aistleitner et al. in \cite{AMMP}. This question was initially posed in \cite{AMMP} following Theorem 1.2, and numerous attempts have been made to address it over the years.

In our previous work \cite{DM}, we established an $\Omega$-result of Bondarenko--Seip type for the quadratic Dirichlet $L$-function $L(1/2, \chi_P)$ over irreducible polynomials $P$ associated with hyperelliptic curves of genus $g$ over a fixed finite field $\mathbb{F}_q$ in the large genus limit. These $L$-functions over function fields constituted the first family deviating from the full orthogonality property, where we focused on irreducible polynomials to avoid certain obstacles.

The main objective of this paper is to investigate the maximum sizes of $\{L(\sigma, \chi_d): |d|\leq X\}$ for any fixed $\sigma\in [1/2, 1]$. Additionally, for $\sigma \in (1/2, 1]$, we aim to evaluate the proportion of fundamental discriminants that fall near the expected order of magnitude up to the precise values of constants.

Throughout this paper, $\sum^{\flat}$ and $\log_j$ represent the summation over fundamental discriminants and the $j$-th iterated natural logarithm, respectively. The notation $\square$ indicates a perfect square. Also, $\varepsilon$ is an arbitrarily small positive quantity which is not necessarily the same at each occurrence. Let $\mbb{N}$ be the set of natural numbers and $\mbb{P}$ be the set of primes.

\subsection{Large values of $L(1/2, \chi_d)$}
 Soundararajan \cite[Theorem 2]{Sound} showed that 
\[
\max_{\substack{d \text{ fundamental discriminant }\\X<|d|\leq 2X}}|L(1/2, \chi_d)|\geq \exp\left(\left(1/\sqrt{5}+o(1)\right)\sqrt{\frac{\log X }{\log_2 X}}\right).
\]

Note the typo in the statement of Theorem 2 in \cite{Sound}. In this article, we use the long resonator technique of Bondarenko--Seip to study the large value of $L(1/2, \chi_d)$  under the Generalized Riemann Hypothesis (GRH), estimating character sums effectively to absorb long resonators within them. Our approach yields the following theorem, improving Soundararajan's above result.    
\begin{theorem}\label{main theorem 1}
 For sufficiently large $X$, under GRH, we have
	\[
	\max_{\substack{d \text{ fundamental discriminant }\\X<|d|\le 2X}}|L(1/2, \chi_{d})|\geq \exp\left(\left(\frac 1 2 +o(1) \right)\sqrt{\frac{\log X \log_3 X}{\log_2 X}}\right).
	\]
\end{theorem}

\subsection{Large values of $L(1, \chi_d)$}
The distribution of $L(1, \chi_d)$ offers an intriguing connection to various arithmetic entities, such as the class number of related quadratic number field. Granville--Soundararajan studied the distribution of values of $L(1, \chi_d)$ and established \cite[Theorem 4]{GS} that the function $\Phi_X(\tau)$, the proportion of fundamental discriminants $d$ satisfying $|d|\leq X$ with $L(1, \chi_d)\geq e^{\gamma}\tau$ has the following doubly exponentially decaying:
\begin{align}\label{gs decay}
\Phi_X(\tau)=\exp\left(-\frac{e^{\tau-C_1}}{\tau}\left(1+O\left(\frac{1}{A}+\frac{1}{\tau}\right)\right)\right)
\end{align}
 uniformly in the range $\tau\leq \log_2 X+\log_4 X-\log_2 A-20$ unconditionally and $\log_2 X+\log_3 X-\log_2 A-20$ under GRH, where $2\leq A\leq \log_2 X$ and $C_1$ is given by \eqref{c1}. This result gives an asymptotic formula for $\Phi_X(\tau)$ in \eqref{gs decay} if $A = A(X)$ is chosen as an arbitrary function of $X$ such that $A(X)\to \infty$ as $X \to \infty$. In \cite{GS}, the authors' studied the distribution of $L(1, \chi_d)$ using a probabilistic model $L(1, Y)$ for a suitable random variable $Y$ and solved one of the conjectures posed by Montgomery--Vaughan \cite[Conjecture 1]{MV}. 
 
 Moreover, Granville and Soundararajan predicted that if the asymptotic formula \eqref{gs decay} holds to the edge of the viable range, i.e., $\tau=\log_2 X+\log_3 X+C_1$ then 
 \[
 \max_{\substack{d \text{ fundamental discriminant }\\|d|\leq X}}L(1, \chi_d)=e^{\gamma}\left(\log_2 X+\log_3 X+C_1+o(1)\right),
 \]
 where $C_1$ is derived from the underlying probabilistic model, specifically
 \begin{align}\label{c1}
 C_1=\int_{0}^1 \left(\tanh y\right)\frac{dy}{y}+\int_1^\infty (\tanh y-1)\frac{dy}{y}=0.8187\ldots.
 \end{align}
 
 Montgomery--Vaughan \cite[Conjecture 2]{MV} also proposed the following conjecture.
\begin{conjecture}[Montgomery--Vaughan]\label{MV con} The number of fundamental discriminants $|d|\leq X$ with $L(1, \chi_d)\geq e^{\gamma}\left(\log_2 |d|+ \log_3 |d|\right)$ is $>X^{\theta}$ and $<X^{\Theta}$, where $0<\theta< \Theta<1$.
	\end{conjecture} 
Assuming the GRH, Granville and Soundararajan proved the upper bound part of the Conjecture \ref{MV con} as a consequence of \cite[Theorem 4]{GS}. Also, the authors' \cite[Theorem 5a]{GS} provided a lower bound towards the conjecture by showing that there are $\gg X^{1/2}$ primes $\leq X$ such that 
\begin{align}\label{gs under primes}
L(1, \chi_p)\geq e^{\gamma}\left(\log_2 X +\log_3 X -\log(2\log 2)+o(1)\right).
\end{align}
   
  In this paper, we employ a long resonator technique to show the following ``existence" and ``proportion" of large values of $|L (1, \chi_d)|$, as $d$ varies over all fundamental discriminants in a range $|d|\leq X$. 
 \begin{theorem}\label{main theorem 2}
 Let $X$ be sufficiently large, and suppose GRH holds. Then there exists a fundamental discriminant $d$ with $|d|\leq X$ such that
 	\[
 |L(1, \chi_{d})|\geq e^{\gamma}\left(\log_2 X +\log_3 X-C_2+o(1)\right),
 \]
 where
 \[
 C_2=\frac{\pi}{4}+\frac{\log 2}{2}+\log\left(3\log 2-\frac{\pi}{2}\right)\approx 0.455967.
 \]
 \end{theorem}
Notice that $-C_2\approx -0.455967$ is slightly weaker than $-\log(2\log 2)\approx -0.3266$ in \eqref{gs under primes}, which is anticipated given that we are considering the full family. However, the following lower bound of the proportion must be compared to the double exponential decay of \eqref{gs decay} as explained in Remark \ref{1st remark}. This result provides that the double exponential decay of type \eqref{gs decay} might persist up to the edge of the viable range, $\log_2 X+\log_3 X+C$, for some constant $C>0$. Such information cannot be uniformly deduced from Theorem 5a of \cite{GS} within a constant factor of the expected order of magnitude $\log_2 X+\log_3 X$. More precisely, if \eqref{gs decay} holds for $\tau=\log_2 X+\log_3 X+C_1-\eta$ for some $\eta>0$, then our results are optimal up to a precise value of constant.
\begin{theorem}\label{main theorem 2a}
 Assume GRH holds. For every $\eta>0$, we define $\Phi_X(\eta)$ to be the proportion of fundamental discriminants $d$ with $|d| \leq X$, for which $|L(1, \chi_d)| > e^{\gamma}\tau_{\eta, X}$, where $\tau_{\eta, X}=\log_2 X+\log_3 X-C_2-\eta$, i.e., 
 \[
 \Phi_X(\eta):=\frac{1}{\sideset{}{^\flat}\sum_{|d|\le X}1} \sideset{}{^\flat}\sum_{\substack{|d|\le X\\ |L(1, \chi_d)|> e^{\gamma}\tau_{\eta, X}}}1.
 \]
 Then we have,
 \[
 \Phi_X(\eta)\geq X^{-\frac{e^{-\eta}}{2}\left(1+O\left(\frac{1}{\sqrt{\log_2 X}}\right)\right)}.
 \]
 \end{theorem}
 
 \begin{remark}\label{1st remark}
   Proportionally, shifting to the left of $-C_2$ by a constant allows us to establish a lower bound on the proportion of fundamental discriminants that decay double exponentially, uniformly dependent on the parameter 
  $\eta$ from Theorem \ref{main theorem 2}. For example, setting $\eta=0.044$ yields $\Phi_X(\eta)\geq X^{-0.478}$ according to Theorem \ref{main theorem 2}. Moreover, if the asymptotic \eqref{gs decay} holds for $\tau=\log_2 X+\log_3 X-C_1-C_2-\eta$ to the quality in Theorem \ref{main theorem 2}, then $\Phi_X(\eta):=\Phi_X(\tau)\approx X^{-0.279e^{-\eta}}$, which is better than our proportion $X^{-0.5e^{-\eta}}$ from Theorem \ref{main theorem 2}. The expression \eqref{gs decay} fails to discern this proportion, given that $A=A(X)\rightarrow \infty$.
     We must mention that for the other families examined in \cite{AMM, AMMP}, the value of $C_2$
     exceeds 1. However, by employing a sieving method on perfect squares, we reduce this value to $0.455967$.
 	\end{remark}
 
 \begin{remark}\label{remark 2}
The subfamily $\{L(1, \chi_p) : p \leq X\}$ achieves large values, both in terms of the constant and proportion, comparable to those established in Theorem \ref{main theorem 2}. The proof follows the same path as in the proof of Theorem \ref{main theorem 2} with a more robust estimate on character sum over primes \cite[Theorem 5.15]{MV} compared to Lemma \ref{character sum estimate}.
 	\end{remark}
 
 \subsection{Large values of $L(\sigma, \chi_d)$ when $\sigma\in (1/2, 1)$} In this range, based on Diophantine approximation, Montgomery \cite{Mont} established that for large $T$,
 \begin{align}\label{mont bound}
 \max_{t\in [T, 2T]}\log|\zeta(\sigma+it)|\geq c(\sigma) (\log T)^{1-\sigma}(\log_2 T)^{-\sigma},
 \end{align}
 where $c(\sigma)=\frac{(\sigma-1/2)^{1/2}}{20}$ unconditionally and $c(\sigma)=1/20$ under RH. The constant $c(\sigma)$ under RH was later improved by Yang \cite{Yang} using the long resonator technique of Bondarenko--Seip. Also, a prediction on $c(\sigma)$ was proposed by Lamzouri \cite[Remark 2]{Lam}. 
 
In \cite{Lam}, Lamzouri studied the logarithm of the absolute value of $L$-functions at $1/2 < \sigma < 1$ across various families, showing it can reach values as large as $(\log X)^{1-\sigma}/\log_2 X$, where $X$ is the conductor of the family. This result can also be attained using Soundararajan's resonance method \cite{Sound}. However, proving Montgomery's $\Omega$-result \eqref{mont bound} for other families of $L$-functions remains challenging. 
 
For $1/2<\sigma<1$, Lamzouri \cite[Theroem 1.8]{Lam} achieved Montgomery's $\Omega$-result for the subfamily $\{L(\sigma, \chi_p): p\leq X\}$ under GRH, using an idea by Granville and Soundararajan \cite[Theorem 5a]{GS}. More precisely, there are $\gg X^{1/2}$ primes $p\leq X$ such that
\begin{align}\label{lam sigma}
\log |L(\sigma, \chi_p)|\geq \left(\beta(\sigma)+o(1)\right)(\log X)^{1-\sigma}(\log_2 X)^{-\sigma},
\end{align}
where $\beta(\sigma)=(2\log 2)^{1-\sigma}/(1-\sigma)> \sqrt{2/\log 2}\approx 1.698$.

Using the long resonator technique, we extend this result to the full family of fundamental discriminants and establish the following theorem, recovering the constant found in Lamzouri's above result. 
  \begin{theorem}\label{main theorem 3}
 	 Assume GRH holds. Let $1/2<\sigma<1$, and $0<b<1$ be given. For sufficiently large $X$, there exists a fundamental discriminant $d$ with $|d|\leq X$ such that
 	\[
 	\log |L(\sigma, \chi_{d})|\geq \left(\alpha(\sigma, b)+o(1)\right)(\log X)^{1-\sigma}(\log_2 X)^{-\sigma},
 	\]
 	where $\alpha(\sigma, b)=\frac{b}{1-\sigma}\left(\alpha(b)\right)^{\sigma-1}$ and $\alpha(b)=2\log \left(\frac{(1+b)^2}{1+b^2}\right)$.
 	\end{theorem}
 Observe that when $b$ approaches $1$ in Theorem \ref{main theorem 3}, we have $a(\sigma, b) > \sqrt{2/\log 2} \approx 1.698$, also found in \eqref{lam sigma}. Here, $b$ represents the coefficient of the resonator at primes.
 Furthermore, the following theorem establishes a lower bound for the proportion within the expected range of uniformity $\asymp_{\sigma} (\log X)^{1-\sigma}(\log_2 X)^{-\sigma}$, except the optimal value of constant that may depend on $\sigma$. This result should be comparable to Theorem 6 in \cite{Lam}, as discussed in Remark \ref{remark 4}.
 \begin{theorem}\label{main theorem 3a}
 Let $X, b, \alpha(b), \alpha(\sigma, b)$ be as in Theorem \ref{main theorem 3}. Assume GRH holds. For every $0<\eta<1/\alpha(b),$ let $\Phi_X(\sigma, \eta)$ be the proportion of fundamental discriminants $d$ such that $|d|\leq X$ and $\log |L(\sigma, \chi_d)|>\tau_{\eta, X}$, where $\tau_{\eta, X}=\alpha(\sigma, b)\left(1-\eta \alpha(b)\right)^{1-\sigma}(\log X)^{1-\sigma}(\log_2 X)^{-\sigma}$. That is
 	\[
 	\Phi_X(\sigma, \eta)=\frac{1}{\sideset{}{^\flat}\sum_{|d|\le X}1} \sideset{}{^\flat}\sum_{\substack{|d|\le X\\ \log |L(\sigma, \chi_d)|> \tau_{\eta, X}}}1.
 	\]
 	Then we have,
 	\[
 	\Phi_X(\sigma, \eta)\geq X^{-\frac12\left(1-\eta \alpha(b)\right)\left(1+O\left(1/\sqrt{\log_2 X}\right)\right)}.
 	\]
 	\end{theorem}
 
 \begin{remark}\label{remark 4}
	 Theorem 6 in \cite{Lam} establishes the existence of a constant $c_2(\sigma)>0$ such that, uniformly in the range $\tau\leq c_2(\sigma)(\log X\log_4 X)^{1-\sigma}/\log_2 X$,
	\begin{align}\label{lam result}	
\Phi_X(\sigma, \tau)=\exp\left(-A_2(\sigma)\tau^{\frac{1}{1-\sigma}}(\log \tau)^{\frac{\sigma}{1-\sigma}}\left(1+O\left(1/\sqrt{\log \tau}+r(y, \tau)\right)\right)\right),
	\end{align}
	where $r(y, \tau), A_2(\sigma)$ are defined as in the corresponding theorem.
	
Theorem \ref{main theorem 3a} provides a method to generate a lower bound for the proportion uniformly in an extended range $\tau< \alpha(\sigma, b)
 (\log X)^{1-\sigma}/(\log_2 X)^{\sigma}$ compared to the range in \eqref{lam result}. Note that this matches with the expected order of magnitude $(\log X)^{1-\sigma}/(\log_2 X)^{\sigma}$. By selecting $\eta$ very close to $1/\alpha(b)$, we can recover \eqref{lam result} except for a different constant $A_2(\sigma)$ and a different error term instead of
  $1/\sqrt{\log_2 X}$. Specifically, this can be achieved by taking $\eta=1/\alpha(b)-\log_4 X/\log_2 X$. We omit the details when $\eta$ is close to $1/\alpha$ in the theorem for the sake of simplicity. 
 
 Moreover, the theorem can be doable to the subfamily $\{L(\sigma, \chi_p): p \leq X\}$ using the same technique, and the number of discriminants can be quantified as in Theorem \ref{main theorem 3a}, following the approach outlined in Remark \ref{remark 2}.
 	\end{remark}
 
 \subsection{Further Discussion}
 The arguments presented in Lemma \ref{character sum estimate} exhibit broad applicability, with potential analogues extendable to various other families of $L$-functions. For instance, cubic $L$-functions over rationals (see \cite{DDLL}) and quadratic Dirichlet $L$-functions in the hyperelliptic ensemble (see \cite{Lum}) are notable examples. Furthermore, the methodologies employed in proving Theorems \ref{main theorem 1}--\ref{main theorem 3} can be effectively adapted to diverse families of interest. One such application is the determination of extreme values of $L$-functions associated with cusp forms, as discussed in Section 4 of \cite{Sound}.
 
 \subsection{Strategy of the paper}
 {\bf For $\bm{ \sigma=1/2}$:} We employ the resonance technique, as outlined in \cite{BS}, focusing on sub-GCD sums rather than full GCD sums as described in \cite{BT}. This entails analyzing $\sum_{|d|\leq X}L(1/2, \chi_d)R_d^2$ instead of $\sum_{|d|\leq X}L(1/2, \chi_d)^2 R_d^2$, because of the non-square contributions from the characters, the sum significantly increases in its second moment. Here, $R_d$ is known as the resonator of $L(1/2, \chi_d)$. In fact, the resonator $R_d$ is chosen
 as a Dirichlet polynomial, and it is often suitable to choose a function with multiplicative coefficients. It can be written as a finite Euler product. Bondarenko--Seip approach achieves superior results by introducing the concept of a ``long resonator," unlike Soundararajan's methods.
 
The main idea is to decompose the sum $\sum_{|d|\leq X}L(1/2, \chi_d)R_d^2$ into components, distinguishing between square and non-square parts arising from the character sum. The character sum accumulates positive mass from the square terms where positivity applies, and we employ the GRH to take control of the error term by establishing bounds within the zero-free region of the corresponding $L$-functions. This is achieved using Lemma \ref{character sum estimate}, where $f(n_0)$ and $g(n_1)$ play a crucial role in mitigating the impact of the lengthy resonators. 
Any sufficiently small exponent applied to the conductor in the error term of the character sum in Lemma \ref{character sum estimate} will disrupt the efficacy of the long resonator technique. 

{\bf For $\bm{1/2< \sigma \le 1}$:} In this interval, we use the resonator $R_d$ as a short Euler product in an utterly multiplicative way because the corresponding $L$-function has a precise and effective short Euler product approximation. It is important to emphasize that the analysis of the resonator at the central point is distinct, as the resonator is not completely multiplicative. 

The remaining analysis follows the same approach described in the second paragraph of the strategy for $\sigma=1/2$, with the additional task of handling non-square-free integers when bounding the functions $f$ and $g$. For $\sigma=1$, we count the exact number of ways the product of three integers forms a square to optimize the constant $C_2$. Reducing $C_2$ is challenging due to the difficulty of approximating the resonator coefficients of the form $r_p=1-p/z$. For $\sigma\in(1/2, 1)$, we select the resonator coefficient as a constant close to $1$ throughout the segment.

\vspace{3mm}
\textbf{Acknowledgment:} 
The authors thank {K. Soundararajan, K. Seip, Y. Lamzouri, C. Aeistleitner, P. Moree, O. Ramar\'{e}, S. Drappeau, and Allysa Lumley} for valuable comments and suggestions. 
The majority of the work was carried out during the first author's postdoctoral tenure at NTNU, facilitated by the European Research Consortium for Informatics and Mathematics Alain Bensoussan Fellowship Programme, supported by the Research Council of Norway Grant 275113. Concurrently, the second author, during a postdoc at the Institut de Mathmatique de Marseille, received funding through the joint FWF-ANR project Arithrand: FWF: I 4945-N and ANR-20-CE91-0006. Presently, both authors are recipients of funding from the Max-Planck Institute for Mathematics.
 \section{A useful Lemma on character sum} 

This section establishes a crucial estimate regarding character sums that are pivotal for proving our main theorems. 
We establish the following character sum using an approximation lemma for the logarithm of Dirichlet $L$-functions.  
\begin{lemma}\label{character sum estimate}
	Assume GRH holds. 
	Let $n=n_0 n_1^2$ be a positive integer with $n_0$ square-free part of $n$. Then for any $\ve>0$, we obtain
	\begin{align*}
	\sideset{}{^\flat}\sum_{|d|\le X} \chi_{d}(n)=\frac{X}{\zeta(2)}\prod_{p|n}\left(\frac{p}{p+1}\right) \mathds{1}_{n=\square}+ O\left(X^{1/2+\varepsilon}f(n_0)g(n_1)\right),
	\end{align*}
	where  $f(n_0)=\exp\left((\log n_0)^{1-\ve}\right)$ arises from the square-free component, while $g(n_1)=\sum_{d\mid n_1}\frac{\mu^2(d)}{b^{1/2+\ve}}$ originates from the square component of the modulus, and $\mathds{1}_{n=\square}$ indicates the indicator function of the square numbers. 
\end{lemma}
\begin{proof}
	Since fundamental discriminants $d$ are either $d\equiv 1 \pmod{4}$, $d$ square-free, or $d=4N, N\equiv 2,3 \pmod{4}$, $N$ square-free, we can confine ourselves to the case $d\equiv 1 \pmod{4}$ and the other cases can be handled similarly.
	
	The contribution of the numbers $d\equiv 1 \pmod{4}$ to the sum is equal to
	\begin{align}\label{contri}
	&=\frac12 \sum_{\psi_4\, \text{(mod 4)}}\sum_{\substack{d\leq X\\ (d, n_1)=1}}\mu^2(d)\psi_4(d)\chi_{n_0}(d)\nonumber\\
	&=\frac12 \sum_{\psi_4\, \text{(mod 4)}}\sum_{\substack{d\leq X}}\psi_4(d)\chi_{n_0}(d)\sum_{\substack{a^2\mid d\\(a, n_1)=1}}\mu(a)\sum_{b\mid (d, n_1)}\mu(b) \nonumber\\
	&=\frac12 \sum_{\psi_4\, \text{(mod 4)}}\sum_{b\mid n_1}\mu(b)\sum_{\substack{a^2 b\leq X\\(a, n_1)=1}}\mu(a)\sum_{d\leq \frac{X}{a^2 b}}\psi_4(d)\chi_{n_0}(d),
	\end{align}
	where $\psi_4$ is a character of modulus $4$ and the two inner sums involving the M\"{o}bius function come from the condition that $d$ is square-free and $(d, n_1)=1$.  To estimate the following character sum: 
	\[
	S(Z; n_0):=\sum_{d\leq Z}\chi(d),
	\] 
	where $\chi=\psi_4 \chi_{n_0}$ is a quadratic character of modulus $4n_0$ and $Z:=X/(a^2 b)$, we can start with the case $n_0>1$. Then, $\chi$ is always a non-principal character. Applying the Perron's formula, for every $\varepsilon>0$,
	\[
	S(Z; n_0)=\frac{1}{2\pi i}\int_{(1+\varepsilon)}L(s, \chi)Z^s\frac{ds}{s}=\frac{1}{2\pi i}\int_{1+\varepsilon-iT}^{1+\varepsilon+iT}L(s, \chi)Z^s\frac{ds}{s}+O\left(\frac{Z^{1+\varepsilon}}{T}\right),
	\]
	where
	\[
	L(s, \chi)=\sum_{n=1}^{\infty}\frac{\chi(n)}{n^s}.
	\]
	We now extend an argument in the proof of Lemma 4.5 in \cite{GS}.
	Using Lemma 8.2 of \cite{GS2}, if $\sigma_1=\min\left(\sigma_0+\frac{1}{\log y}, \frac{\sigma+\sigma_0}{2}\right)$, where $y\geq 2, |t|\geq y+3$, and $\sigma>\sigma_0\geq 1/2$, then 
	\begin{align}\label{approxi of logL(sigma)}
	\log L(s, \chi)=\sum_{n=2}^y \frac{\Lambda(n)\chi(n)}{n^s \log n}+O\left(\frac{\log q}{(\sigma_1-\sigma_0)^2}y^{\sigma_1-\sigma}\right).
	\end{align}
	For $|t|\leq T$ and $y=(\log T)^{\beta}$, we get
	\[
	|\log L(s, \chi)|\ll y^{1-\sigma}+\frac{\log T}{(\sigma_1-\sigma_0)^2}y^{\sigma_1-\sigma}\ll (\log T)^{\beta(1-\sigma)}+(\log T)^{1+\beta(\sigma_0-\sigma)}\log_2 T.
	\]
	Inserting this bound after shifting the line of integration in the above Perron's integral to $\Re(s)=\sigma>\sigma_0\geq \frac12$,
	\[
	S(Z; n_0)\ll Z^{\sigma} \exp\left((\log T)^{\beta(1-\sigma)}+(\log T)^{1+\beta(\sigma_0-\sigma)}\log_2 T\right)+\frac{Z^{1+\varepsilon}}{T}.
	\]
	Due to the periodicity of characters, we may suppose that $Z\leq 4n_0$. Also equating $\beta(1-\sigma)$ and $1+\beta(\sigma_0-\sigma)$, we can choose $\sigma_0=\frac{\beta -1}{\beta}$. Take $T=n_0$. Therefore, for $\sigma>\sigma_0\geq 1/2$,
	\[
	S(Z; n_0)\ll Z^{\sigma} f_{\sigma}(n_0),
	\]
	where $f_{\sigma}(n_0)=\exp\left((\log n_0)^{\beta(1-\sigma)}\right)$.
	 
	Plugging in these estimates, for $\sigma\geq 1/2+\varepsilon$, we obtain
	\[
	\eqref{contri} \ll X^{\sigma}f_{\sigma}(n_0)\sum_{b\mid n_1}\frac{\mu^2(b)}{b^{\sigma}}\sum_{\substack{a^2 b\leq X\\ (a, n_1=1)}}\frac{\mu^2(a)}{a^{2\sigma}}\ll X^{\sigma}f_\sigma(n_0)g_\sigma(n_1),
	\]
 where $g_\sigma(n_1)=\sum_{b\mid n_1}\frac{\mu^2(b)}{b^{\sigma}}$.
	
	When $n_0=1$, we isolate the principal and non-principal characters from the outermost sum, allowing us to manage the non-principal characters using the method outlined above. Therefore, in this case, the contribution of the numbers $d\equiv 1 \pmod{4}$ to the sum equals 
	\begin{align*}
	&=\frac{X}{2}\sum_{b\mid n_1}\frac{\mu(b)}{b}\sum_{\substack{a^2 b\leq X\\ (a, n_1)=1}}\frac{\mu(a)}{a^2}+O\left(X^{\sigma}f_\sigma(n_0)g_\sigma(n_1)\right)\\
	&=\frac{X}{2\zeta(2)}\prod_{p\mid n_1}\left(\frac{p}{p+1}\right)+O\left(X^{\sigma}f_\sigma(n_0)g_\sigma(n_1)\right).
	\end{align*}
	In particular, taking $\sigma=\frac{1}{2}+\varepsilon$ and $\beta=2$, we complete the proof.
\end{proof}


\section{Proof of the Theorem \ref{main theorem 1}}
We begin with defining the resonator for the $L$-functions at the central point, along with a few necessary lemmas. Using these lemmas, we will complete the proof of the theorem.
  
{\it The Resonator for $L(\frac 1 2, \chi_{d})$:}
The resonator is the Dirichlet polynomial 
\[
R_d:= \sum_{m\in \mathcal{R}}\psi(m)\chi_d(m),
\]with $|\mathcal{R}|\leq X^{\theta}$ for some $\theta\leq 1/2$, ensuring the coefficients $\psi(h)$ are non-negative and resonate with $L(1/2, \chi_d)$. 

Following section 2 of \cite{BS2}, the set $\mathcal{R}$ is constructed as follows: 

Let $N$ be a large number that will be chosen later. Let $a\in (1, \infty)$ and $\delta\in (0,1)$ be fixed real numbers.
Let $\mathcal{P}$ be the set of all primes $p$ such that $$e\log N \log_{2} N< p\le e^{(\log_{2} N)^{\delta}} \log N \log_{2} N.$$ 
Let us define a multiplicative function $\psi$ supported on the set of square-free numbers such that for any  $p\in\mathcal{P}$,
\begin{align*}
\psi(p)=\sqrt{\frac{\log N \log_{2}N }{\log_{3}N}}p^{-1/2}\left(\log p-\log(\log N \log_{2}N)\right)^{-1}.
\end{align*}
Let $\mathcal{P}_{k}$ be the set of all primes $p$ such that $$e^{k}\log N \log_{2} N < p\le e^{k+1}\log N \log_{2} N,$$
where $k=1,\ldots,[(\log_{2}N)^{\delta}]$. 
Fix $1<a<\frac{1}{\delta}$. Then, let $ \mathcal{R}_{k}$ be the set of positive integers numbers that have at least $\frac{a\log N}{k^{2}\log_{3}N}$ prime divisors in $\mathcal{P}_{k}$, and let $ \mathcal{R}^{\prime}_{k}$ be the set of integers from $\mathcal{R}_{k}$ that having all their prime divisors in $\mathcal{P}_{k}$. Finally set 
\[
\mathcal{R}:=\rm{supp}(\psi)\setminus\displaystyle \cup_{k=1}^{[(\log_{2}N)^{\delta}]}\mathcal{R}_{k}.
\]
\begin{lemma}\label{bound for a typical n}
	If $n\in\mathcal{R}$ then
	\[
	n\ll \exp\left(\frac{\log N \log_{2} N}{\log_{3}N}\right).
	\]
	\end{lemma}
\begin{proof}
Observe that, $\mathcal{R}$ is the set of square-free numbers that have at most $\frac{a\log N}{k^{2}\log_{3}N}$ prime divisors in each set $\mathcal{P}_{k}$.
So, if $n\in\mathcal{R}$, then 
\begin{align*}
\log n\le \sum_{k=1}^{[(\log_{2}N)^{\delta}]} \frac{a\log N}{k^{2}\log_{3}N} \log\left(e^{k+1}\log N \log_{2} N\right) \ll \frac{\log N \log_{2} N}{\log_{3}N},
\end{align*}
which finishes the proof.
\end{proof}	
The resonator $R_d$ is simpler than those in \cite{BS} and \cite{BT}, whereas, in these papers, it is required to construct a subset $\mathcal{R}'\subset \mathcal{R}$ to manage local factors such as $\log (m/n)$ for the zeta function and specific congruence classes for Dirichlet L-functions in terms of conductor aspects. In our case, $\mathcal{R}'$ coincides with $\mathcal{R}$, as for any $l\in \mathcal{R}'$,
\[
r(l)^2=\sum_{\substack{h\in \mathcal{R}\\ hl=\square}}\psi(h)^2=\psi(l)^2 \iff r(l)=\psi(l),
\] 
where $r(l)$ is a non-negative function defined on $\mathcal{R}$.
\begin{lemma}\label{Le1} 
	Suppose that $1<a<\frac{1}{\delta}$. Then, for $N$ large enough, we have $|\mathcal{R}|\leq N$.
\end{lemma}

\begin{proof}
	This is essentially the first part of the proof presented in Lemma 2 of \cite{BS2}.
\end{proof}
Let us define 
\begin{align*}
\mcal{A}_{N}:=\frac{1}{\sum_{m\in\mbb{N}} \psi(m)^{2} }\sum_{n\in\mbb{N}}\frac{\psi{(n)}}{\sqrt{n}}h(n)\sum_{l|n}\psi(l)\sqrt{l},
\end{align*}
where $h(1)=1$ and $h(n)=\prod_{p\mid n}\frac{p}{p+1}$, for $n\ge 2$, determines the orthogonality mass of the associated family. 
Since $\psi$ and $h$ are multiplicative functions, we obtain 
\begin{align}\label{eq2}
\mcal{A}_{N}=\prod_{p\in\mathbb{P}}\frac{1+\psi(p) h(p)p^{-1/2} +   \psi(p)^2 h(p) }{1+\psi(p)^2 }.
\end{align}

\begin{lemma}\label{Le2}
	As $N\to\infty$, we have
	\[\mcal{A}_{N}\ge \exp\left((\delta + o(1))\sqrt{\frac{\log N\log_{3}N}{\log_{2}N}}\right).
	\]	
\end{lemma}

\begin{proof}
	Since $\psi(p)<(\log _3 N)^{-1/2}$ for all $p\in \mathbb{P}$, it follows from \eqref{eq2} that 
	\begin{align*}
	\mcal{A}_{N}&=\prod_{p\in\mathbb{P}}\frac{1+\psi(p) h(p)p^{-1/2} +   \psi(p)^2 h(p) }{1+\psi(p)^2 }= \exp\left( (1+o(1))\sum_{p\in \mathbb{P}}\frac{\psi(p)}{\sqrt{p}}\left(\frac{p}{p+1}\right) \right)\\
	&=\exp\left( (1+o(1))\sum_{p\in \mathbb{P}}\frac{\psi(p)}{\sqrt{p}} \right).
	\end{align*}
	The remainder proof follows from Lemma 1 of \cite{BS2}.
	
\end{proof}  

\begin{lemma}\label{Le3}
	As $N\to\infty$, we have 
	\begin{align*}
	\frac{1}{\sum_{m\in\mbb{N}} \psi(m)^{2}}\sum_{n\in\mbb{N};\,  n\notin\mathcal{R}}\frac{\psi{(n)}h(n)}{\sqrt{n}}\sum_{l|n}\psi(l)\sqrt{l}=o(\mcal{A}_{N}).
	\end{align*}	
	
\end{lemma}  

\begin{proof}
	Since $0<h(n)\le 1$, we see that 
	\begin{align*}
	\frac{1}{ \sum_{m\in\mbb{N}} \psi(m)^{2}}\sum_{n\in\mbb{N};\,  n\notin\mathcal{R}}\frac{\psi{(n)}h(n)}{\sqrt{n}}\sum_{l|n}\psi(l)\sqrt{l}\le \frac{1}{ \sum_{m\in\mbb{N}} \psi(m)^{2}}\sum_{n\in\mbb{N};\,  n\notin\mathcal{R}}\frac{\psi{(n)}}{\sqrt{n}}\sum_{l|n}\psi(l)\sqrt{l}.
	\end{align*}
 The proof is now completed on invoking  Lemma 2 of \cite{BS2}.

\end{proof} 

\begin{lemma}\label{Le5}
	Let $\ve$ be a positive real number. Then as $N\to\infty$, 
	\[ 
	\frac{1}{\sum_{m\in\mbb{N}} \psi(m)^{2}} \sum_{n\in\mbb{N}}\frac{\psi{(n)}h(n)}{\sqrt{n}}\sum_{\substack{l|n\\ l\le n/N^{\varepsilon}}}\psi(l)\sqrt{l}=o(\mcal{A}_{N}),
	\]
	where the implicit constant only depends on $\varepsilon$.
\end{lemma}

\begin{proof} Similar to the  proof of Lemma \ref{Le3}.
	Since $0<h(n)\le 1$, we have 
	\begin{align*}
	\frac{1}{\sum_{m\in\mbb{N}} \psi(m)^{2}} \sum_{n\in\mbb{N}}\frac{\psi{(n)}h(n)}{\sqrt{n}}\sum_{\substack{l|n\\ l\le n/N^{\ve}}}\psi(l)\sqrt{l}\le \frac{1}{\sum_{m\in\mbb{N}} \psi(m)^{2}} \sum_{n\in\mbb{N}}\frac{\psi{(n)}}{\sqrt{n}}\sum_{\substack{l|n\\ l\le n/N^{\ve}}}\psi(l)\sqrt{l}.
	\end{align*}
	 The proof is now completed on invoking  Lemma 3 of \cite{BS2}.
	
\end{proof}

\begin{lemma}[Approximate Functional Equation]\label{Approximate Functional Equation}
	Let $\chi_{d}$ be a quadratic Dirichlet character. Then 
	\begin{align*}
	L(1/2,\chi_{d})&=2\sum_{n}\frac{\chi_{d}(n)}{\sqrt{n}} U\left(\frac{n}{\sqrt{d}} \right), 
	\end{align*} 	 
	where $$U(x)=\frac{1}{2\pi i}\int_{(c)} \pi^{-s/2}\frac{\Gamma\left(\frac{s}{2}+\frac{1}{4}\right)}{\Gamma(\frac{1}{4})}x^{-s}\frac{ds}{s}, $$ with $c>0$. The weight function $U(x)$ is real-valued, smooth on $(0, +\infty)$, bounded as $x$ approaches $0$ and decays exponentially as $x\to +\infty$. More precisely, $U(x)$ satisfies $U(x)= 1 + O(x^{\frac{1}{2}-\varepsilon})$ for small $x$, and $U(x)\ll e^{-x} $ for large $x$. Moreover the derivative $U(x)$ satisfies $U^{'}(x)\ll x^{\frac{1}{2}-\varepsilon}e^{-x} $.
	\end{lemma}

\begin{proof} See the proofs of Lemma 2.1 and 2.2 of \cite{Sound0}.

\end{proof}

\begin{proof}[Proof of Theorem \ref{main theorem 1}]
	  We introduce
\begin{align*} 
 \mcal{S}_{1}:=\sideset{}{^\flat}\sum_{X<d\le 2X}L\left(\frac{1}{2},\chi_{d}\right) R_d^2, \quad \text{ and } \quad \mcal{S}_{2}:=\sideset{}{^\flat}\sum_{X<d\le 2X }R_d^2.
\end{align*}
 Thus, we have
\[ \max_{\substack{d \text{ fundamental discriminant }\\X<d\le 2X}} \big|L(1/2,\chi_{d})\big|\ge \frac{\mcal{S}_1}{\mcal{S}_2}. 
\]
To begin, we calculate a lower bound for $\mcal{S}_{1}$. Using Lemma \ref{Approximate Functional Equation}, we have
\begin{align*}
	\mcal{S}_{1}=&\sideset{}{^\flat}\sum_{\substack{X<d\le 2X}} L\left(\tfrac{1}{2},\chi_{d}\right)R_d^2 \\
	=&2\sum_{m, n\in \mcal{R}} \psi(m) \psi(n)\sum_{l\geq 1} \frac{1}{\sqrt{l}}\sideset{}{^\flat}\sum_{X<d\le 2X}\chi_{d}(lmn) U\left(\frac{l}{\sqrt{d}} \right).     
\end{align*}
Divide the innermost expression into its square and non-square components,  according to whether $lmn=\square$ and $lmn\neq \square$, respectively. Applying Lemma \ref{character sum estimate} alongside the upper bound of $U(x)$ yields	
\begin{align*}	
    \mcal{S}_{1}=\frac{2X}{\zeta(2)}
     \sum_{m, n \in \mcal{R}}&\psi(m)\psi(n)\sum_{\substack{l\geq 1\\ lmn =\square}} \frac{1}{\sqrt{l}}\prod_{p\mid lmn}\left(\frac{p}{p+1}\right)\int_{1}^{2}U\left(\frac{l}{\sqrt{Xt}} \right) dt\\
       &+O\Bigg(X^{1/2 +\varepsilon}\sum_{l\ge 1}\frac{e^{-l/\sqrt{X}}}{\sqrt{l}} \sum_{\substack{m, n \in \mcal{R}\\lmn=s_0 s_1^2\\ \mu(s_0 )\ne 0}}\psi(m) \psi(n) f(s_0)g(s_1)\Bigg).
	\end{align*}
	To manage the error term effectively, it is crucial to understand the significance of functions $f(s_0)$ and $g(s_1)$, as they effectively incorporate the lengthy resonators.
 
Now, we use Lemma \ref{bound for a typical n} to bound
\begin{align}\label{bound on f at 1/2}
f(s_0)=\exp\left((\log s_0)^{1-\varepsilon}\right)\ll \exp\left(\left(\frac{\log N \log_2 N}{\log_3 N}\right)^{1-\varepsilon}\right)\ll N^{\varepsilon/2}.
\end{align}

Since $g$ is a multiplicative function, in fact, $g(mn)\leq g(m)g(n) \, \forall\, m,n \geq 1$, the condition $lmn=s_0 s_1^2$ gives $g(s_1)\leq g(s_0 s_1^2)\leq g(l)g(m)g(n)$. Also from the resonator's construction, we note that if $m\in \mathcal{R}$ and $p \mid m$, then $p\leq e^{(\log_{2} N)^{\delta}} \log N \log_{2} N=:M$. So, the prime number theorem leads to
\begin{align}\label{bound on g at 1/2}
g(m)=\sum_{d\mid m}\frac{\mu^2(d)}{d^{1/2+\varepsilon}}\ll \exp\left(\sum_{p\leq M}\frac{1}{p^{1/2+\varepsilon}}\right)\ll \exp\left(\frac{M^{1/2-\varepsilon}}{(1/2-\varepsilon)\log M}\right)\ll N^{\varepsilon/2}.
\end{align}
Hence, applying the Cauchy--Schwarz inequality and above bounds \eqref{bound on f at 1/2} and \eqref{bound on g at 1/2} on $f(s_0)$ and $g(m)$ respectively, the error term of $\mcal{S}_{1}$ is bounded above by
	\begin{align}\label{bound for s2}
	&\ll X^{1/2+\varepsilon}\left(\sum_{l\ge 1}\frac{e^{-l/\sqrt{X}}g(l)}{\sqrt{l}}\right)\left(\sum_{m\in \mathcal{R}}\psi(m)g(m)\right)^2 \nonumber\\
	&\ll X^{1/2+\varepsilon}N^{2\varepsilon}|\mcal{R}|\left(\sum_{l\ge 1}\frac{e^{-l/\sqrt{X}}g(l)}{\sqrt{l}}\right)\left(\sum_{m\in \mathcal{R}}\psi(m)^2\right).
	\end{align}
	Due to exponential decay it is enough to restrict the above $l$-sum up to $X^{1/2+\varepsilon}$.  Since $g(l)=\prod_{p|l}\left(1+\frac {1}{p^{1/2 +\ve}}\right)\ll l^{\ve}$,  so we get 
\begin{align*}
\sum_{l\ge 1}\frac{e^{-l/\sqrt{X}}g(l)}{\sqrt{l}}\ll X^{\ve} \sum_{l\le X^{1/2 +\ve}}\frac{1}{\sqrt{l}}\ll  X^{1/4 +\ve}.
\end{align*}

	Therefore, we have
	\[
	\eqref{bound for s2}\ll X^{3/4+2\varepsilon}N^{2\varepsilon}|\mcal{R}|\left(\sum_{m\in \mathcal{R}}\psi(m)^2\right).
	\]
Since the coefficients of resonator $\psi(n)$ are non-negative and the smooth function $U(x)$ is also positive, picking the term only when $lm=n$, we end up with
\begin{align*}
\mcal{S}_{1}&  
	 \ge\frac{2X}{\zeta(2)}\sum_{n \in\mcal{R}} \frac{\psi(n)}{\sqrt{n}} \prod_{p|n}\left(\frac{p}{p+1}\right) \sum_{\substack{m \mid n}} \psi(m)\sqrt{m}\int_{1}^{2}U\left(\frac{n}{m\sqrt{Xt}} \right) dt\\
	 &+O\left(X^{3/4+2\varepsilon}N^{2\varepsilon}|\mcal{R}|\left(\sum_{m\in \mathcal{R}}\psi(m)^2\right)\right)\\
	& \ge  \frac{2X}{\zeta(2)}\sum_{n \in\mcal{R}} \frac{\psi(n)}{\sqrt{n}} h(n)\sum_{\substack{m\mid n\\ m\ge n/N^{\varepsilon}}} \psi(m)\sqrt{m} +O\left(X^{3/4 +2\varepsilon}N^{1+2\varepsilon} \left(\sum_{\substack{m\in \mcal{R}}}\psi(m)^{2}\right)\right),
\end{align*}
where in the second inequality we use $U(x)=1 + O\left(x^{\frac{1}{2}-\varepsilon}\right)$, the definition of $h(n)$ and $|\mathcal{R}|\leq N$.  Take 
$N\asymp X^{(\frac{1}{4}-5\varepsilon)}$, where $0<\varepsilon<1/20$ is arbitrary small. We use Lemma \ref{Le2}, Lemma \ref{Le3}, and Lemma \ref{Le5} to get
\begin{align*}
	 \mcal{S}_{1}&\ge
	\frac{2X}{\zeta(2)} \exp\left((\delta+ o(1))\sqrt{\frac{\log N\log_{3}N}{\log_{2}N}}\right)\left(\sum_{\substack{m\in \mcal{R}}}\psi(m)^{2}\right) +O\left(X^{1-\varepsilon}\left(\sum_{\substack{m\in \mcal{R}}}\psi(m)^{2}\right)\right)\\
		 & \ge \left(1+ o(1) \right) \frac{X}{\zeta(2)} \exp\left(\left(\delta \sqrt{1/4-5\varepsilon}+o(1)\right)\sqrt{\frac{\log X\log_{3} X}{\log_{2} X}}\right) \left(\sum_{\substack{m\in \mcal{R}}}\psi(m)^{2}\right).
\end{align*}
Now, using the arguments outlined above, we proceed to find an upper bound of $\mcal{S}_2$. Again, by applying Lemma \ref{character sum estimate}, the Cauchy--Schwarz inequality, alongside bounds \eqref{bound on f at 1/2} and \eqref{bound on g at 1/2} on $f$ and $g$ respectively, we have
\begin{align*}
	\mcal{S}_{2}&=\sideset{}{^\flat}\sum_{X<d\le 2X}R_d^2=\sum_{m, n \in \mcal{R}}\psi(m) \psi(n)\sideset{}{^\flat}\sum_{X<d\le 2X}\chi_{d}(m n)\\
	            &= \frac{X}{\zeta(2)} \sum_{\substack{m, n \in \mcal{R} \\ m n =\square}}\psi(m) \psi(n)\prod_{p\mid mn}\left(\frac{p}{p+1}\right) 
	             +O\Bigg(X^{\frac 1 2 + \varepsilon} \sum_{\substack{m,n\ge 1\\mn=s_0 s_1^2\\\mu(s_0 )\ne 0}}\psi(m) \psi(n)f(s_0)g(s_1)\Bigg)\\
	            &=\frac{X}{\zeta(2)}	\sum_{\substack{n \in \mcal{R}}}\psi(n)^{2}h(n) + O\left(X^{1/2 +\varepsilon} N^{2\varepsilon} |\mcal{R}|\sum_{\substack{n\in \mcal{R}}}\psi(n)^{2} \right),
\end{align*}
since $m$ and $n$ are square-free with $mn=\square$ implies that $m=n$.	
Therefore, the above choice of $N\asymp X^{\frac{1}{4}-5\ve }$ and $h(n)\leq 1$ gives
\begin{align*}            
  \mcal{S}_{2} \le \left(1+o(1)\right)\frac{X}{\zeta(2)}
	            \left(\sum_{\substack{n\in \mcal{R}}}\psi(n)^{2}\right).
\end{align*}
Hence, for sufficiently large $X$ and arbitrary small $\ve>0$, we get

 \begin{align*}
\max_{\substack{d \text{ fundamental discriminant }\\X<d\le 2X}} \big|L(1/2,\chi_{d})\big|&\geq \frac{\mcal{S}_1}{\mcal{S}_2}\geq \exp\left(\left(\delta \sqrt{1/4-5\varepsilon}+o(1)\right)\sqrt{\frac{\log X\log_{3} X}{\log_{2} X}}\right).
\end{align*}
The proof is concluded by taking $\delta$ to be $1$.
\end{proof}

\section{Proof of Theorem \ref{main theorem 2}}

Unlike at the central point, $L(1, \chi_d)$ has a well-behaved Euler product approximation that allows us to focus our study on its truncated Euler product.

Let $L(1,\chi_d;y):= \prod_{p\le y}$$\left(1-\frac{\chi_d(p)}{p}\right)^{-1} = \sum_{k\ge 1}\frac{a_k \chi_{d}(k)}{k}$ where $a_k =1$ if all prime divisors of $k\le y$, and $0$ otherwise. The following lemma, essentially due to Elliott \cite{Et}.

\begin{lemma}\label{short Euler product}
	Let $y=(\log X)^{180}$.  For all but at most $X^{1/2}$ fundamental discriminants $|d|\le X$, we have
	\begin{align}\label{asym for L(1)}
		L(1,\chi_{d})=L(1,\chi_d;y)\left(1+O\left(\frac{1}{(\log X)^{10}}\right)\right).
	\end{align}
	\end{lemma}
Under GRH, one may employ a more refined short Euler product, as obtained in Proposition 8.2 of \cite{GS}, though for our purpose, Lemma \ref{short Euler product} suffices.

We begin the proof by taking $\mathcal{E}$ to be the set of exceptional discriminants $d$ with $|d|\leq X$ for which the asymptotic in \eqref{asym for L(1)} does not hold. Let us consider, $z=\frac{1}{c}\log X \log_2 X$, where $c>0$ can be optimized later. 
Motived from \cite{AMM, AMMP, HIL}, define resonator in the following way: 
 \[R_{d}=\prod_{p\le z} \left(1-\left(1-\frac{p}{z}\right)\chi_{d}(p)\right)^{-1}=\sum_{k\ge 1} r_k \chi_{d}(k),
 \] 
 where $r_{{p}} =(1-p/z)$ for $p< z$, $r_p=0$ for $p\geq z$. Additionally, $r_k$ can determined in a completely multiplicative manner, which is crucially advantageous compared to the necessity of square-free supports for the resonators when dealing with $L(1/2, \chi_d)$.

We start with the following sums:
\begin{align*}	
	&S_1 := \sideset{}{^\flat}\sum_{\substack{|d|\le X\\ d\notin \mcal{E}}} L(1,\chi_d;y) R_{d}^{2} =\sideset{}{^\flat}\sum_{\substack{|d|\le X}} L(1,\chi_d;y) R_{d}^{2} - \sideset{}{^\flat}\sum_{\substack{|d|\le X\\ d\in \mcal{E}}} L(1,\chi_d;y) R_{d}^{2}\\	
	&S_2 := \sideset{}{^\flat}\sum_{\substack{|d|\le X\\ d\notin \mcal{E}}} R_{d}^{2} =\sideset{}{^\flat}\sum_{\substack{|d|\le X}}  R_{d}^{2} - \sideset{}{^\flat}\sum_{\substack{|d|\le X\\ d\in \mcal{E}}} R_{d}^{2}.
\end{align*}
Note that, $|\mcal{E}|\le X^{1/2}$ and under GRH, $|L(1,\chi_d;y)|\ll \log_{2} X$. 
We use the inequality $|1-|a||\le |1-a|$ to bound
\begin{align*}|R_{d}|=\prod_{p\le z} \big|\left(1-r_p \chi_{d}(p)\right)^{-1}\big| \le \prod_{p\le z} \left( 1-\left(1- \frac{p}{z}\right) \right)^{-1}\le \prod_{p\le z}\frac{z}{p}.
\end{align*} 
So, the prime number theorem gives
\begin{align}\label{RD}
	&\log|R_{d}|\le \log z\sum_{p\le z} 1 -\sum_{p\le z}\log p =\frac{z}{\log z}\left(1+O\left(1/\log z\right)\right)\nonumber\\ 
	&\implies |R_{d}|\le e^{\frac{z}{\log z}\left(1+O\left(1/\log z\right)\right)}\leq X^{\frac{\left(1+O\left(1/\log z\right)\right)}{c}}.
\end{align}
 Then
\begin{align*}
	\sideset{}{^\flat}\sum_{\substack{|d|\le X\\ d\in \mcal{E}}} L(1,\chi_d;y) R_{d}^{2}\ll X^{\frac12 +\frac2c+\frac{\ve}{5}}\log_2 X \ll X^{\frac12+\frac2c+\frac{\ve}{4}}.
\end{align*}
Therefore, we obtain
 \begin{equation*}
	\frac{S_1}{S_2}=\frac{\sideset{}{^\flat}\sum_{\substack{|d|\le X\\ d\notin \mcal{E}}} L(1,\chi_d;y) R_{d}^{2}}{\sideset{}{^\flat}\sum_{\substack{|d|\le X\\ d\notin \mcal{E}}} R_{d}^{2}}=\frac{\sideset{}{^\flat}\sum_{\substack{|d|\le X}} L(1,\chi_d;y) R_{d}^{2} +O\left(X^{\frac12+\frac2c+\frac{\ve}{4}}\right)}{\sideset{}{^\flat}\sum_{\substack{|d|\le X}}  R_{d}^{2} +O\left(X^{\frac12+\frac2c+\frac{\ve}{4}}\right)}.
\end{equation*}
Hence, we work with the following full sums: 
 \[S_1^*=\sideset{}{^\flat}\sum_{\substack{|d|\le X}} L(1,\chi_d;y) R_{d}^{2} \quad \text{ and } \quad S_2^*=\sideset{}{^\flat}\sum_{\substack{|d|\le X}}R_{d}^{2}.
 \]

We begin with calculating $S_2^*$. Using Lemma \ref{character sum estimate}, we get
 \begin{align*}
	&S_2^*=\sideset{}{^\flat}\sum_{\substack{|d|\le X}}R_{d}^{2}=\sum_{m,n\ge 1}r_m r_n \sideset{}{^\flat}\sum_{\substack{|d|\le X}}\chi_{d}(mn)\\
	&=\sum_{\substack{m, n \ge 1 \\ mn=\square}}r_m r_n \Bigg(\frac{X}{\zeta(2)}\prod_{\substack{p|mn }} \left(\frac{p}{p+1}\right)  + O\left(X^{\frac12+\ve} g(mn) \right)\Bigg) +O\Bigg(X^{\frac 1 2 + \ve} \sum_{\substack{m, n, s_0, s_1\ge 1\\mn=s_0 s_1^2\\\mu(s_0 )\ne 0}}r_m r_nf(s_0)g(s_1)\Bigg)  \\
	&=\frac{X}{\zeta(2)}\sum_{\substack{m, n \ge 1\\ mn=\square }}r_m r_n h(mn)  + O\Bigg(X^{\frac12+\ve} \sum_{\substack{m, n\ge 1\\ mn=\square }}r_m r_n g(mn) \Bigg)\\
	& +O\Bigg(X^{\frac 1 2 + \ve} \sum_{\substack{m, n, s_0, s_1\ge 1\\mn=s_0 s_1^2\\\mu(s_0 )\ne 0}}r_m r_n f(s_0)g(s_1)\Bigg).	
\end{align*}
As we discussed in the proof of Theorem \ref{main theorem 1}, the crux of the matter lies in the significance of functions $f(s)$ and $g(s)$, which essentially absorb the lengthy resonators and make meaningful contributions to ensure manageable error terms.
In fact, $p\mid mn$ implies $p\le z$. So, the prime number theorem leads to
\begin{align}\label{bound on g}
g(k)=\sum_{d\mid k}\frac{\mu^2(d)}{d^{1/2+\varepsilon}}\ll \exp\left(\sum_{p\leq z}\frac{1}{p^{1/2+\varepsilon}}\right)\ll \exp\left(\frac{z^{1/2-\varepsilon}}{(1/2-\varepsilon)\log z}\right)\ll X^{\varepsilon/8}.
\end{align}
The same argument can be used to bound $g(s_1)$ inside the second error term. On the other hand, since $n_0$ is a square-free part of $mn$ and $s_0\mid mn$, it implies that $s_0\le \prod_{p\le z} p = e^{z(1+o(1))}$, so that 
$
f(s_0)=\exp\left((\log s_0)^{1-\ve}\right)\ll X^{\varepsilon/4}.
$ 
Using these estimates and \eqref{RD}, the error terms are of size $O\left( X^{\frac12+\frac2c+ \frac{11\varepsilon}{8}}\right)$.
 Therefore, 
\begin{align*}
	S_2^* &= \frac{X}{\zeta(2)}\sum_{\substack{m, n\\ mn=\square }}r_m r_n h(mn) +O\left( X^{\frac12+\frac2c+ \frac{11\varepsilon}{8}} \right)\\
	&=\frac{X}{\zeta(2)}\sum_{\substack{k\ge 1} } r^2_k\, d(k^2) h(k)+O\left( X^{\frac12+\frac2c+ \frac{11\varepsilon}{8}} \right).
\end{align*}
Again, using Lemma \ref{character sum estimate}, we get
\begin{align*}
	S_1^*&=\sideset{}{^\flat}\sum_{\substack{|d|\le X}}L(1,\chi_d; y) R_{d}^{2}=\sum_{m,n\ge 1}r_m r_n \sum_{l\ge 1}\frac{a_l}{l} \sideset{}{^\flat}\sum_{\substack{|d|\le X}}\chi_{d}(lmn)\\
	&=\sum_{l\ge 1}\frac{a_l}{l}\sum_{\substack{m,n\ge 1\\lmn=\square}}r_m r_n \Bigg(\frac{X}{\zeta(2)}\prod_{\substack{p|lmn }} \left(\frac{p}{p+1}\right) + O\left(X^{\frac12+\ve}g(lmn) \right)\Bigg)\\
	&+O\Bigg(X^{\frac12 + \ve}\sum_{l\ge 1}\frac{a_l}{l} \sum_{\substack{m, n\ge 1\\lmn=s_0 s_1^2\\ \mu(n_0 )\ne 0}}r_m r_n f(s_0)g(s_1)\Bigg).
\end{align*}
Recall that, $g$ is multiplicative and in fact $g(ab)\leq g(a)g(b)$ for every $a, b\in \mathbb{N}$.  So, $ g(lmn)\leq g(l)g(m)g(n)$. Therefore, the first error term of the above expression is bounded above by
$$\ll X^{1/2+\varepsilon} \left(\sum_{l\ge 1}\frac{a_l g(l)}{l}\right) \left(\sum_{\substack{m\ge 1}}r_m g(m)\right)^2.$$
Since $y=(\log X)^{180}$, we have
$$\displaystyle\sum_{l\ge 1}\frac{a_l g(l)}{l}\ll \prod_{p\le y}\left(1+\frac{g(p)}{p}\right)\ll \log y\ll\log_{2} X\ll X^{\ve/4}.$$ 
Also, noting $g(p^k)=g(p)$ for any $k\geq 1$, and using prime number theorem, we have
\begin{align}\label{rmgm}
\displaystyle\sum_{\substack{m\ge 1}} r_m g(m) & =\prod_{p< z}\left(1+\frac{r_p g(p)}{1-r_p}\right)=\prod_{p< z}\left(\frac{z}{p}\left(1+\frac{1}{p^{1/2+\varepsilon}}-\frac{p^{1/2-\varepsilon}}{z}\right)\right)\nonumber\\
&=\exp\left(\sum_{p< z}\log \left(\frac{z}{p}\right)+\sum_{p<z}\log\left(1+\frac{1}{p^{1/2+\varepsilon}}-\frac{p^{1/2-\varepsilon}}{z}\right)\right)\nonumber\\& = X^{\frac{1+o(1)}{c}}.
\end{align} 
Therefore, the first error term is $O\left( X^{\frac12+\frac2c+\frac{3\varepsilon}{2}}\right)$.  We move into the second error term, which is trickier due to the extra involvement of the non-multiplicative function $f$. Using the inequality $(a+b)^x\leq a^x+b^x$ for $a, b\ge 0$ and $x\in (0, 1]$, we obtain
$f(s_0)\leq f(l_0 m_0 n_0)\leq f(l_0)f(m_0)f(n_0)$, where $l_0, m_0, n_0$ are the square-free part of $l, m, n$ respectively. Also, we have $g(s_1)\leq g(s_0s_1^2)\leq g(l)g(m)g(n)$. Applying these estimates, the second error term is bounded above by
\begin{align}\label{end et}
 \ll X^{1/2+\varepsilon}\left(\sum_{l\ge 1}\frac{a_l f(l_0)g(l)}{l}\right) \left(\sum_{\substack{m\ge 1}}r_m f(m_0)g(m)\right)^2.
 \end{align}
Using the bound \eqref{bound on g} on $g(m)$, we compute
\begin{align*}
&\sum_{l\ge 1}\frac{a_l f(l_0)g(l)}{l}\le \sum_{l\le X}\frac{a_l f(l)g(l)}{l} + \sum_{l> X}\frac{a_l f(l) g(l)}{l}\\
&	\le f(X)\sum_{l\ge 1}\frac{a_l g(l)}{l} + \sum_{l> X}\frac{a_l l^{1/(\log X)^{1-\ve}}g(l)}{l}\\
&\ll X^{\frac{3\ve}{20}} + \prod_{p\le y}\left( 1+ \frac{1}{p^{1-(\log X)^{-\ve}}}\right)\ll\exp\left(\frac{(\log X)^{\ve} y^{(\log X)^{-\ve} }}{\log y} \right)\ll X^{\frac{3\ve}{20}}.
\end{align*}
From \eqref{rmgm} and above estimation together with the bound $f(m_0)\ll X^{\ve/4}$, we get
\[
\eqref{end et}\ll  X^{\frac12+\frac2c+\frac{19\varepsilon}{10}}.
\]
Hence, 
\begin{align*}
	S_1^*=\frac{X}{\zeta(2)} \sum_{l\ge 1}\frac{a_l}{l}\sum_{\substack{m,n\ge 1\\ lmn=\square }}r_m r_n h(lmn)  + O\left(X^{\frac12+\frac2c+\frac{19\varepsilon}{10}}\right).
\end{align*}
Consider $L(1,\chi_d; z):= \prod_{p\le z}$$\left(1-\frac{\chi_d(p)}{p}\right)^{-1} = \sum_{k\ge 1}\frac{c_k \chi_{d}(k)}{k}$, where $c_k =1$ if all prime divisors of $k\le z$, and $0$ otherwise. Here, we employ an additional essential criterion for the positivity of resonator coefficients, namely $r_m\geq 0$, alongside ensuring that $a_l\geq c_l\geq 0$, enabling us to disregard prime factors larger than $z$. Consequently, we arrive at
\begin{align*}
	S_1^*\ge \frac{X}{\zeta(2)} \sum_{l\ge 1}\frac{c_l}{l}\sum_{\substack{m,n\ge 1\\ lmn=\square }}r_m r_n h(lmn)  + O\left(X^{\frac12+\frac2c+\frac{19\varepsilon}{10}} \right).
\end{align*}
Now, from the above estimates of $S_1^*$ and $S_2^*$, we want to evaluate an exact expression for the following sums to optimize the constant $c$:
\begin{align}\label{eq M}
M_{c, z}:=\sum_{l\ge 1}\frac{c_l}{l}\sum_{\substack{m,n\ge 1\\ lmn=\square }}r_m r_n h(lmn) \quad \text{ and } \quad S_{c, z}:=\sum_{\substack{n\ge 1 }}r_n^2 d(n^2)h(n).
\end{align}
This is perhaps difficult for other families where the condition $l\mid n$ is applied and throw other values of $l$ from the $l$-sum by using positivity.

We decompose $l=l_1 l_2^2 l_3^2$ where $l_1$ and $l_2$  are square-free, with $(l_1 , l_2) = 1$, and $p|l_3 \implies p|l_1 l_2$, i.e., $l_1$ is the product of all primes dividing $l$ to an odd power, and $l_2$ is the product of all primes dividing $l$ to an even power $\ge 2$.  Similarly, we decompose $m$ and $n$ as $m=m_1m_2^2m_3^2$ and $n=n_1n_2^2n_3^2$, where $m_i$ and $n_i$ $(i=1, 2)$ are  square-free, $(m_1, m_2)=(n_1, n_2)=1$, and for any prime $p$: $p\mid m_3\implies p\mid m_1 m_2$, and $p\mid n_3\implies p\mid n_1 n_2$. 

So $lmn=\square \implies l_1 m_1 n_1 =\square$, where $ l_1 , m_1,$ and $ n_1$ are square free integers.

Without loss of generality, let $u=(l_1 , n_1)$. Write $l_1 = u l_1^{\prime}$ and $n_1 = u n_1^{\prime}$ with $(l_1^{\prime}, n_1^{\prime})=1$. Therefore $l_1 m_1 n_1=\square \implies m_ 1=  l_1^{\prime} n_1^{\prime}$.  Then $M_{c, z}$ becomes
\begin{align*}
M_{c, z}&=\sum_{\substack{u\geq 1\\ (u,\,  l_1^{\prime}  n_1^{\prime} l_2 n_2)=1}} \frac{ \mu^2(u) c_u r_u }{u} \sum_{\substack{ l_1^{\prime}\geq 1\\ ( l_1^{\prime},\,   n_1^{\prime} l_2 m_2)=1}} \frac{  \mu^2( l_1^{\prime}) c_{l_1^{\prime}} r_{ l_1^{\prime}}}{ l_1^{\prime}}  \sum_{\substack{ n_1^{\prime}\geq 1\\ ( n_1^{\prime},\, n_2 m_2)=1}} \mu^2( n_1^{\prime}) r_{ n_1^{\prime}}^2  \\
&\times \sum_{l_2 \geq 1}\frac{\mu^2(l_2) c_{l_2 }}{l_2^2}   \sum_{m_2 \geq 1}\mu^2(m_2) r_{m_2}^2 
\sum_{n_2 \geq 1}\mu^2(n_2) r_{n_2}^2 h(u l_1^{\prime} n_1^{\prime} l_2 m_2 n_2)\\ 
&\times \sum_{\substack{l_3\geq 1 \\q|l_3  \implies q|u l_1^{\prime}  l_2 }} \frac{ c_{l_3}  }{l_3^2 } \sum_{\substack{m_3\geq 1\\ q\mid m_3 \implies q| l_1^{\prime} n_1^{\prime} m_2 }} r_{m_3}^2  \sum_{\substack{n_3\geq 1 \\q|n_3 \implies q| u n_1^{\prime} n_2 }} r_{n_3}^2.
\end{align*}
Define, $f_p =\left(1-\frac{1}{p^2}\right)^{-1}$ and $g_p =\left(1-r_{p}^2\right)^{-1}$. We express the following sums into the Euler product form
\begin{align*}
\sum_{\substack{l_3\geq 1\\ q|l_3  \implies q|u l_1^{\prime}  l_2 }} \frac{ c_{l_3}  }{l_3^2 }= \prod_{\substack{q\mid u   l_1^{\prime}  l_2  }} f_p, \hspace{2cm}   \sum_{\substack{m_3\geq 1\\ q|m_3 \implies q| l_1^{\prime} n_1^{\prime} m_2 }} r_{m_3}^2 =\prod_{\substack{q\mid  l_1^{\prime} n_1^{\prime} m_2  }} g_p,    
\end{align*}
and \[
 \sum_{\substack{n_3\geq 1 \\q|n_3 \implies q| u n_1^{\prime} n_2 }} r_{n_3}^2 =\prod_{\substack{q\mid  u n_1^{\prime} n_2 }} g_p.
\]
 Therefore, 
\begin{align*}
M_{c, z}&=\sum_{\substack{u\geq 1\\ (u,\,  l_1^{\prime}  n_1^{\prime} l_2 n_2)=1}} \frac{ \mu^2(u)c_u r_u }{u} \sum_{\substack{ l_1^{\prime}\geq 1\\ ( l_1^{\prime},\,   n_1^{\prime} l_2 m_2)=1}} \frac{  \mu^2(l_1^{\prime}) c_{l_1^{\prime}} r_{ l_1^{\prime}}}{ l_1^{\prime}}  \sum_{\substack{ n_1^{\prime}\geq 1\\ ( n_1^{\prime},\,    n_2 m_2)=1}} \mu^2( n_1^{\prime}) r_{ n_1^{\prime}}^2  \sum_{l_2\geq 1 }\frac{\mu^2(l_2) c_{l_2 }}{l_2^2} \\  
&\times \sum_{m_2 \geq 1}\mu^2(m_2) r_{m_2}^2 
\sum_{n_2 \geq 1}\mu^2(n_2) r_{n_2}^2  h(u l_1^{\prime} n_1^{\prime} l_2 m_2 n_2) \prod_{\substack{q\mid u   l_1^{\prime}  l_2  }} f_p \prod_{\substack{q\mid  l_1^{\prime} n_1^{\prime} m_2  }} g_p \prod_{\substack{q\mid  u n_1^{\prime} n_2 }} g_p .
\end{align*}
By multiplicativity this transform into the Euler product form 
\begin{align*}
M_{c, z} &=\prod_{p< z} \left( 1+ \frac{2r_p}{p} f_p g_p  h(p) + \frac{2r_p^3}{p} f_p g_p^2 h(p) + r_p^2 g_p^2 h(p) + \frac{r_p^2}{p^2} f_p g_p^2 h(p) +\frac{1}{p^2} f_p h(p)    \right.\\
&\quad \quad\quad\quad\quad\quad\quad\quad\quad\quad  \left.  + \frac{2r_p^2}{p^2} f_p g_p h(p) + \frac{r_p^4}{p^2} f_p g_p^2 h(p) + 2r_p^2 g_p h(p) +r_p^4 g_p^2 h(p)   \right) \\
&=\prod_{p< z}  f_p g_p^2 h(p) \left( 1+r_p^2+ \frac{2 r_p}{p} +\frac{1}{p} \left(1-r_p^2\right)^2 \left(1-\frac{1}{p^2}\right)\right)\\
&=A(z) \prod_{p< z}\left(1-\frac{1}{p} \right)^{-1} h(p)  \prod_{p< z}\left(1-r_{p}^2 \right)^{-2} \left(1+r_{p}^2 \right),
\end{align*}
where 
\begin{align}\label{A(z)}
A(z):=\prod_{p< z}\left(1- \frac{(1-r_p)^2}{(p+1)(1+r_p^2)}+\left(1-\frac{1}{p}\right)\frac{(1-r_p^2)^2}{p(1+r_p^2)} \right).
\end{align}
On the other hand, $r_n$ is completely multiplicative and  $h(n)=\prod_{p\mid n}\frac{p}{p+1}$ is multiplicative such that $h(p^k)=h(p)$ for any $k\geq 1$. So,
\begin{align*}
S_{c, z}&=\prod_{p<z}\left(1+3r_p^2 h(p)+5r_p^4 h(p)+\cdots\right)=\prod_{p<z}\left(1-r_p^2 \right)^{-2} \left(1+ \frac{p-2}{p+1} r_p^2 + \frac{r_p^4}{p+1}   \right)\\
&=B(z)\prod_{p<z}\left(1-r_p^2 \right)^{-2} \left(1+r_p^2\right),
\end{align*}
where we use $d(p^k)=k+1$ for all $k\geq 1$, and
\begin{align}\label{B(z)}
B(z):=\prod_{p<z}\left(1-\frac{3r_p^2}{(p+1)\left(1+r_p^2\right)}+\frac{r_p^4}{(p+1)\left(1+r_p^2\right)}\right).
\end{align}
From the estimation of $M_{c, z}$ and $S_{c, z}$, we finally obtain 
\begin{align}\label{ratio}
	\frac{S_1^*}{S_2^*}\ge \frac{\frac{X}{\zeta(2)} A(z) \prod_{p< z}  \left(1-\frac{1}{p} \right)^{-1}h(p)  \prod_{p< z}\left(1-r_{p}^2 \right)^{-2} \left(1+r_{p}^2 \right)  + O\left(X^{\frac12+\frac2c+\frac{19\varepsilon}{10}} \right)}{\frac{X}{\zeta(2)}  B(z)\prod_{p<z}\left(1-r_p^2 \right)^{-2} \left(1+r_p^2\right)  + O\left(X^{\frac12+\frac2c+\frac{11\varepsilon}{8}} \right)},
\end{align}
where $A(z)$ and $B(z)$ are as in \eqref{A(z)} and \eqref{B(z)} respectively. Noting $z=\frac1c \log X \log_2 X$, we first optimize $c$ from the ratio to the right-hand side of \eqref{ratio} by comparing the error terms. To do this, we first calculate the size of $\prod_{p<z}\left(1-r_p^2 \right)^{-2} \left(1+r_p^2\right)$ using partial summation and the prime number theorem. We find that
\begin{align*}
\prod_{p<z}\left(1-r_p^2 \right)^{-2} \left(1+r_p^2\right)&=\exp\left(-2\int_2^z\log \left(1-r_t^2\right)\frac{dt}{\log t}+\int_2^z\log \left(1+r_t^2\right)\frac{dt}{\log t}\right).
\end{align*}
The first integral is referenced from \cite[page $841$]{AMMP}. We evaluate the second integral as
\begin{align*}
\int_2^z\log \left(1+r_t^2\right)\frac{dt}{\log t}&=(1+O(1/\log z))\int_2^z\frac{2r_t}{(1+r_t^2)}\frac{tdt}{\log t}\\
&=(1+O(1/\log z))\frac{z}{\log z}\int_0^1\frac{2u(1-u)}{2-2u+u^2}du,
\end{align*}
by using the prime number theorem and change of variable $t=uz$. So,
\begin{align*}
\prod_{p<z}\left(1-r_p^2 \right)^{-2} \left(1+r_p^2\right)&=\exp\left(\left(2(2-\log 4)+c'+O(1/\log z)\right)\frac{z}{\log z}\right)\\
&=X^{\frac{2(2-\log 4)+c'}{c}\left(1+O(1/\log_2 X)\right)},
\end{align*}
where
$c'=\int_0^1\frac{2u(1-u)}{2-2u+u^2}du=\log 2+\frac{\pi}{2}-2.$ 

Using Mertens' theorem, we obtain $B(z)\asymp 1/\log X$. So, from \eqref{ratio}, we compare the first term of the denominator to the both error terms to restrict ourself
\begin{align}\label{optimzed constant}
\frac12+\frac2c< 1+\frac{2(2-\log 4)+c'}{c}\implies c> 2\left(3\log 2-\frac{\pi}{2}\right)\approx 1.0172.
\end{align}
Also, from the above computations, we obtain
\begin{align}\label{scx}
S_{c, z}=X^{\frac{2(2-\log 4)+c'}{c}\left(1+O(1/\log_2 X)\right)}.
\end{align}

To get main term from the ratio in the right hand side of \eqref{ratio}, we need to estimate $A(z)$, $\prod_{p<z}h(p)^{-1}B(z)$ together with an explicit lower bound in the third theorem of Mertens' \cite[Theorem 8]{RS} as 
\[
\prod_{p< z} \left(1-\frac{1}{p}\right)^{-1}\geq e^{\gamma}\log z \left(1-\frac{2}{\log^2 z}\right).
\]
Again, using partial summation and the prime number theorem, we get
\begin{align*}
A(z)=\exp\left((1+O(1/\log z))\int_2^z\log\left(1- \frac{(1-r_t)^2}{(t+1)(1+r_t^2)}+\left(1-\frac{1}{t}\right)\frac{(1-r_t^2)^2}{t(1+r_t^2)}\right)\frac{dt}{\log t}\right). 
\end{align*}
Here $r_t=1-t/z$. To evaluate the integral inside the exponential, we split the range of integration into two segments: one for
 $t\leq z^{\theta}$ and another for $t> z^{\theta}$, $0<\theta<1$. While $t\leq z^{\theta}$, the contribution to the integral is equal to 
\begin{align*}
&\int_{2}^{z^{\theta}}\log\left(1-\frac{t^2}{z^2(t+1)(2-2t/z+t^2/z^2)}+\left(1-1/t\right)\frac{t(2-t/z)^2}{z^2(2-2t/z+t^2/z^2)}\right)\frac{dt}{\log t}\\
&= (1+o(1))\frac{1}{z^2}\int_2^{z^{\theta}}\frac{tdt}{\log t}\ll \frac{1}{z^{2(1-\theta)}\log z},
\end{align*}
which is $O(1/(\log z)^2)$ for any $\theta<1$. The main contribution comes from the segment $t>z^{\theta}$, in which case, under change of variable $t=zu$, the integral transform into
\begin{align*}
&\frac{z}{\log z}\int_{z^{\theta-1}}^{1}\log\left(1-\frac{u^2}{(zu+1)(2-2u+u^2)}+\left(1-\frac{1}{zu}\right)\frac{u(2-u)^2}{z(2-2u+u^2)}\right)\frac{du}{1+\frac{\log u}{\log z}}\\
&=(1+O(1/\log z))\frac{z}{\log z}\int_{z^{\theta-1}}^{1} \log\left(1-\frac{u}{z(2-2u+u^2)}+\frac{u(2-u)^2}{z(2-2u+u^2)}\right)du\\
&= (c_2-c_3+O(1/\log z))\frac{1}{\log z},
\end{align*}
since $zu+1=zu\left(1+O(z^{-\theta})\right)$ and $\log (1+x)=x+O(x^2)$ for $x<1$, where 
\begin{align}\label{c3}
c_2=\int_0^1\frac{u(2-u)^2}{2-2u+u^2}du, \quad \text{ and } \quad c_3=\int_0^1\frac{u}{2-2u+u^2}du= \frac{\pi}{4}-\frac{\log 2}{2}\approx 0.43882.
\end{align}
Therefore, we deduce that
\[
A(z)=\exp\left((c_2-c_3+O(1/\log z))\frac{1}{\log z}\right).
\]
Following the above arguments, we have
\begin{align*}
B(z)\prod_{p<z}h(p)^{-1}=\prod_{p<z}\left(1+\frac{(1-r_p^2)^2}{p(1+r_p^2)}\right)
=\exp\left((c_2+O(1/\log z))\frac{1}{\log z}\right).
\end{align*}
Therefore, plugging these computations in \eqref{ratio} with $z=\frac{1}{c}\log X \log_2 X$, we obtain
\begin{align}\label{ratio final}
\frac{S_1^*}{S_2^*}\geq e^{\gamma}\left(\log_2 X+\log_3 X-c_3-\log c+O\left(\frac{1}{\log_2 X}\right)\right),
\end{align}
where $c_3$ is given by \eqref{c3}.
Also, from \eqref{optimzed constant}, we can choose $c$ in such a way that 
\begin{align}\label{constant c}
\log c=\log(2(3\log 2-\pi/2))+5\eta-\frac{1}{\sqrt{\log_2 X}}.
\end{align}

The additional factor $-\frac{1}{\sqrt{\log_2 X}}$ in the choice of $c$ is crucial for proving Theorem \ref{main theorem 2a}. We define it this way to maintain uniformity, although any factor of the form $-\frac{1}{(\log_2 X)^{1-\ve}}$ for any $\varepsilon>0$ would work.
Hence, using Lemma \ref{short Euler product} $($see eq. \eqref{asym for L(1)}$)$, we conclude that there exists a fundamental discriminant $d$ with $|d|\leq X$ such that
\[
|L(1,\chi_d)|\ge e^{\gamma}\left(\log_2 X +\log_3 X-C_2-5\eta+o(1)\right),
\]
where 
\[C_2=c_3+\log \left(2\left(3\log 2-\pi/2\right)\right)=\frac{\pi}{4}-\frac{\log 2}{2}+\log \left(2\left(3\log 2-\pi/2\right)\right).\]
 
This completes the proof by choosing $\eta>0$ sufficiently small.

\section{Proof of Theorem \ref{main theorem 2a}}
We will prove the theorem using the estimates derived in the proof of Theorem \ref{main theorem 2}.

 Using \eqref{ratio final}, \eqref{constant c} and Lemma \ref{short Euler product}, we have 
\begin{align*}
\left(\tau_{\eta, X}+\frac{1}{2\sqrt{\log_2 X}}\right)\sideset{}{^\flat}\sum_{\substack{|d|\le X\\ d\notin \mcal{E}}}R_d^2&\leq \bigg|\sideset{}{^\flat}\sum_{\substack{|d|\le X\\ d\notin \mcal{E}}}L(1, \chi_d)R_d^2\bigg|\\
&\leq \bigg|\sideset{}{^\flat}\sum_{\substack{|d|\le X;\, d\notin \mcal{E}\\ |L(1, \chi_d)|\leq \tau_{\eta, X}}}L(1, \chi_d)R_d^2\bigg|+\bigg|\sideset{}{^\flat}\sum_{\substack{|d|\le X;\, d\notin \mcal{E}\\ |L(1, \chi_d)|> \tau_{\eta, X}}}L(1, \chi_d)R_d^2\bigg|\\
&\leq \tau_{\eta, X} \sideset{}{^\flat}\sum_{\substack{|d|\le X\\ d\notin \mcal{E}}}R_d^2+ \bigg|\sideset{}{^\flat}\sum_{\substack{|d|\le X;\, d\notin \mcal{E}\\ |L(1, \chi_d)|> \tau_{\eta, X}}}L(1, \chi_d)R_d^2\bigg|.
\end{align*}
From the above computation of $S_2$ together with \eqref{scx} gives us
\[
\bigg|\sideset{}{^\flat}\sum_{\substack{|d|\le X;\, d\notin \mcal{E}\\ |L(1, \chi_d)|> \tau_{\eta, X}}}L(1, \chi_d)R_d^2\bigg|\geq \frac{e^{\gamma}}{2 \sqrt{\log_2 X}}\sideset{}{^\flat}\sum_{\substack{|d|\le X\\ d\notin \mcal{E}}}R_d^2\geq \frac{e^{\gamma}}{2\zeta(2)}\frac{X^{\left(1+O(1/\log_2 X)\right)\left(1+\frac{2(2-\log 4)+c'}{c}\right)}}{\sqrt{\log_2 X}}.
\]
On the other hand, using a result of Littlewood \cite[Theorem 1]{Little} under GRH that
\[
\max_{|d|\leq X}|L(1, \chi_d)|\leq \left(2 + o(1)\right)e^{\gamma}\log_2 X,
\]
we have from \eqref{RD},
\begin{align*}
\bigg|\sideset{}{^\flat}\sum_{\substack{|d|\le X;\, d\notin \mcal{E}\\ |L(1, \chi_d)|> \tau_{\eta, X}}}L(1, \chi_d)R_d^2\bigg|&\leq \max_{|d|\leq X} R_d^2 \max_{|d|\leq X}|L(1, \chi_d)|\frac{X}{\zeta(2)}\Phi_X(\eta)\\
&\leq 2e^{\gamma}\frac{X}{\zeta(2)}\Phi_X(\eta) X^{\frac{2\left(1+O\left(1/\log_2 X\right)\right)}{c}} \log_2 X.
\end{align*}
Comparing the above lower and upper bound, and using $c$ as expressed in \eqref{constant c}, we finally conclude that
\[
\Phi_X(\eta)\geq \frac{1}{4(\log_2 X)^{3/2}}X^{-\frac{(3\log 2-\pi/2)}{c}\left(1+O\left(1/\log_2 X\right)\right)}\geq X^{-\frac{e^{-5\eta}\left(1+O\left(1/\sqrt{\log_2 X}\right)\right)}{2}},
\]
which finishes the proof.

\section{Proof of Theorem \ref{main theorem 3}}
Since the GRH is assumed, then equation  \eqref{approxi of logL(sigma)} with $t = 0$, $y = (\log X)^{4/(\sigma-1/2)}$ and $\sigma_0 = \sigma/2 + 1/4 > 1/2$, we obtain
\begin{align}\label{initial log expression}
\log L(\sigma,\chi_d)= \sum_{n=2}^{y}\frac{\Lambda(n)\chi_d(n)}{n^{\sigma}\log n} +O\left(\frac{1}{(\log X)^{2}}\right).
\end{align} 
We assume that $1/2<\sigma<1$ is fixed, and set $y = (\log X)^{4/(\sigma-1/2)}$.
Since$$\sum_{k\ge 2}\sum_{n=p^k}^{y}\frac{\Lambda(n)\chi_d(n)}{n^{\sigma}}\ll 1,$$
it follows that 
\begin{align}\label{log expression}
\log L(\sigma,\chi_d)=S_{\chi_d}(\sigma, y) +O(1), 
\end{align}
where $S_{\chi_d}(\sigma, y)=\sum_{p\le y}\frac{\chi_d(p)}{p^{\sigma}}$. Consequently, to prove the theorem, it suffices to exhibit large values of $S_{\chi_d}(\sigma, y)$. To achieve this, we employ a long resonator technique to approximate $S_{\chi_d}(\sigma, y)$, similar to the approach taken in the previous section.

We set $Y=a\log X \log_2 X$, where $a>0$ to be optimized later. Furthermore, we set $r_1 =1 $ and $r_p =0$ for $p>Y$, and $r_p = b$ for all primes $p\le Y$, where $0<b<1$ is fixed. We extend this in a completely multiplicative manner to obtain weights $r_n$ for all $n \ge 1$. Define \begin{align*}
R_{d}=\prod_{p\le Y} \left(1-r_p \chi_{d}(p)\right)^{-1}=\sum_{n=1}^{\infty} r_n \chi_d (n).
\end{align*}
Similar to the previous sections, we consider the sums
\begin{align*}	
S_1(\sigma) :=\sideset{}{^\flat} \sum_{\substack{|d|\le X}} S_{\chi_d}(\sigma, y) R_{d}^{2} 
\quad\quad \text{     and   }\quad\quad S_2(\sigma)  := \sideset{}{^\flat}\sum_{\substack{|d|\le X}} R_{d}^{2}. 
\end{align*}

Therefore 
\begin{align}\label{max schid}
\max_{|d|\le X}|S_{\chi_d}(\sigma, y)|\ge \frac{|S_1 (\sigma)|}{S_2 (\sigma)}.
\end{align}

Applying Lemma \ref{character sum estimate}, we have 
\begin{align*}
S_1 (\sigma) &=\sideset{}{^\flat}\sum_{\substack{|d|\le X}} S_{\chi_d}(\sigma, y) R_{d}^{2}=\sum_{m,n}r_m r_n \sum_{p\le y}\frac{1}{p^{\sigma}}\sideset{}{^\flat}\sum_{\substack{|d|\le X}}\chi_{d}(mnp)\\
&= \sum_{p\le y}\sum_{\substack{m,n \\ mnp=\square } }\frac{ r_m r_n}{p^{\sigma}}\Bigg(  \frac{X}{\zeta(2)}\prod_{\substack{q\,prime\\q\mid mnp}} \left(\frac{q}{q+1}\right)  + O\left(X^{\frac12+\frac{\ve}{3}} g(k) \right)\Bigg)\\
& +O\Bigg( X^{\frac{1}{2} + \frac{\ve}{3}} \sum_{p\le y}\frac{ 1}{p^{\sigma}} \sum_{\substack{m,n\ge 1\\mnp=s_0 s_1^2\\ \mu(s_0 )\ne 0}} r_m r_n f(s_0)g(s_1) \Bigg).
\end{align*}  
The previous section shows that it is enough to prove the second error term, where we use the bounds on $f$ and $g$. Since $g(s_1)\le g(s_0 s_1^2)\le g(m)g(n)g(p)$ and $f(s_0)\ll X^{\ve/3}$, the  second error term is bounded by
 \begin{align*}
\ll X^{1/2+2\varepsilon/3} \left(\sum_{p\le y }\frac{ g(p)}{p^{\sigma}}\right) \left(\sum_{\substack{m\ge 1}}r_m g(m)\right)^2.
\end{align*}
Note that, $r_m$ and $g(m)$ are completely multiplicative and multiplicative functions with $r_p =b$ and $g(p^l)=g(p)=1+\frac{1}{p^{1/2 +\ve}}$ for $l\ge 1$ respectively. Again, using the prime number theorem and the choice of $Y$, we have
\begin{align*}
\sum_{\substack{m\ge 1}}r_m g(m)&=\prod_{p\le Y}\left( 1+ r_p g(p) +r^2_p g(p) +r^3_p g(p)+\ldots \right)\\
&=\prod_{p\le Y}\left( 1+ \frac{r_p g(p)}{ 1- r_p} \right)
=\prod_{p\le Y}\left(\frac{1}{1-b}+\frac{b}{(1-b)p^{1/2 +\ve}} \right)\\
&= \exp\left(-(\log (1-b)+O(1/\log Y)) \frac{Y}{\log Y} \right)=X^{-a\log (1-b)(1+O(1/\log_2 X))}.
\end{align*}
Also, for $1/2<\sigma<1$,
\begin{align*}
\sum_{p\le y }\frac{ g(p)}{p^{\sigma}}=\sum_{p\le y }\frac{ 1}{p^{\sigma}} +\sum_{p\le y }\frac{ 1}{p^{\sigma+1/2+\ve}}=(1+o(1))\frac{y^{1-\sigma}}{(1-\sigma)\log y}\ll X^{\ve/3}.
\end{align*}
Therefore, both the error terms in $S_1(\sigma)$ are $O\left(X^{\frac12-2a\log (1-b)+\varepsilon}\right).$
The condition $mnp=\square$ forces $p\leq Y$. Using positivity of the coefficients $r_m$, and collecting those primes $p$ which divide $n$,
\begin{align*}
 \frac{X}{\zeta(2)} \sum_{p\le y}\sum_{\substack{m,n \\ mnp=\square}}\frac{ r_m r_n}{p^{\sigma}} \prod_{\substack{q\,prime\\q\mid mnp}} \left(\frac{q}{q+1}\right) &\geq  \frac{X}{\zeta(2)}\sum_{p\le Y}\sum_{\substack{m,n \\ mnp=\square\\ p\mid n}}\frac{ r_m r_n}{p^{\sigma}} \prod_{\substack{q\,prime\\q\mid mn}} \left(\frac{q}{q+1}\right)\\
 &\geq \frac{X}{\zeta(2)} \left(\sum_{p\le Y}\frac{ r_p}{p^{\sigma}}\right)\Bigg(\sum_{\substack{m, l\geq 1\\ml=\square}} r_m r_l \prod_{\substack{q\,prime\\q\mid ml}} \left(\frac{q}{q+1}\right)\Bigg),
\end{align*}
where we substitute $n=pl$.
Again, the prime number theorem and the choice of  $Y=a \log X \log_2 X$ yield
\begin{align*}\sum_{p\le Y}\frac{ r_p} {p^{\sigma}}&=\sum_{p\le Y}\frac{ b}{p^{\sigma}}=(1+O(1/\log Y))\frac{b Y^{1-\sigma}} {(1-\sigma)\log Y}\\
&=\frac{ba^{1-\sigma}}{1-\sigma}(1+O(1/\log_2 X))\left( \log X\right)^{1-\sigma}\left( \log_2 X\right)^{-\sigma}.
\end{align*}
Therefore, 
\begin{align*}
S_1 (\sigma)=\frac{ba^{1-\sigma}(1+O(1/\log_2 X))}{1-\sigma}\Bigg(\sum_{\substack{m, l\geq 1\\ml=\square}} r_m r_l h(ml)\Bigg)\frac{X}{\zeta(2)}\left( \log X\right)^{1-\sigma}\left(\log_2 X\right)^{-\sigma}\\
+ O\left(X^{\frac12-2a\log (1-b)+\varepsilon} \right).
\end{align*}
Similarly, again using Lemma \ref{character sum estimate} and following the estimate of $S_1(\sigma)$, we get 
\begin{align*}
S_2 (\sigma) &=\sideset{}{^\flat}\sum_{\substack{|d|\le X}}  R_{d}^{2}=\sum_{m,n}r_m r_n \sideset{}{^\flat}\sum_{\substack{|d|\le X}}\chi_{d}(mn)\\
&= \frac{X}{\zeta(2)}\sum_{\substack{m,n \\ mn=\square}} r_m r_n h(mn)
 +O\Bigg( X^{\frac{1}{ 2} + \ve}\sum_{\substack{m,n\ge 1\\mn=s_0 s_1^2\\ \mu(s_0 )\ne 0}} r_m r_n f(s_0)g(s_1) \Bigg)\\
&= \frac{X}{\zeta(2)}  \sum_{\substack{m,n \\ mn=\square}} r_m r_n h(mn) + O\left(X^{\frac12-2a\log (1-b) +\varepsilon} \right).
\end{align*}
The resonator sum can be evaluated as
\begin{align}\label{resonator length for sigma}
 \sum_{\substack{m,n \\ mn=\square}} r_m r_n h(mn)&=\sum_{\substack{k\ge 1} } r^2_k\, d(k^2) h(k)=\prod_{p\le Y}\left(1-h(p)+h(p)(1-r_p^2)^{-2}(1+r_p^2)\right)\nonumber\\
&=\prod_{p\le Y}\left(\frac{1+b^2}{(1-b^2)^2}\right)\left(1- \frac{1}{p+1}+\frac{(1-b^2)^2}{(1+b^2)(p+1)}\right)\nonumber\\
&= \exp\left(\theta(b)\left(1+O(1/\log Y)\right)\frac{Y}{\log Y}\right)
=X^{a\theta(b)\left(1+O(1/\log_2 X)\right)},
\end{align}
where $\theta(b):= \log \left(\frac{1+b^2}{(1-b^2)^2}\right)$. 

We evaluated this resonator sum to find an optimal value of $a$ as below.
Combining these estimates and comparing the error terms, we finally conclude
\begin{align}\label{s_1/s_2}
\frac{S_1(\sigma)}{S_2(\sigma)}\geq \frac{ba^{1-\sigma}}{1-\sigma}(1+O(1/\log_2 X))\left( \log X\right)^{1-\sigma}\left( \log_2 X\right)^{-\sigma},
\end{align}
provided that
\[
\frac12-2a\log (1-b) < 1+a\,\theta(b) \iff a<\left(2 \log \left(\frac{(1+b)^2}{1+b^2}\right)\right)^{-1}.
\]
By choosing $a=\left(2\log \left(\frac{(1+b)^2}{1+b^2}\right)\right)^{-1}- \eta$ with $\eta>0$ arbitrary small, from \eqref{log expression}, \eqref{max schid} and \eqref{s_1/s_2}, we complete the proof.

Here we must mention that following the same path as of the Theorem \ref{main theorem 2} counting $mnp=\square$ to the sum $S_1(\sigma)$, one can obtain $\frac{2b}{1+b^2}$ instead of $b$ in \eqref{s_1/s_2}. However, while $b$ is close to $1$, both are identical.

\section{Proof of Theorem \ref{main theorem 3a}}
We will employ the strategy to prove Theorem \ref{main theorem 2a} and utilize the estimates derived in the proof of Theorem \ref{main theorem 3}. 

Recall that \[\tau_{\eta, X}=\alpha(\sigma, b)\left(1-\eta \alpha(b)\right)^{1-\sigma}(\log X)^{1-\sigma}(\log_2 X)^{-\sigma},\] where
\[
\alpha(b)=2\log \left(\frac{(1+b)^2}{1+b^2}\right), \quad \alpha(\sigma, b)=\frac{b}{1-\sigma}\left(\alpha(b)\right)^{\sigma-1}.\]
Let us consider $a=1/\alpha(b)-\eta+1/\sqrt{\log_2 X}$ with $\eta >1/\sqrt{\log_2 X}$.
This yields \begin{align*}
a^{1-\sigma}&=\left(\alpha(b)-\eta+\frac{1}{\sqrt{\log_2 X}}\right)^{1-\sigma}=\alpha^{\sigma-1}(b)\left(1-\eta \alpha(b)\right)^{1-\sigma}\left(1+\frac{\alpha(b)}{(1-\eta \alpha(b))\sqrt{\log_2 X}}\right)^{1-\sigma}\\
&\geq \alpha^{\sigma-1}(b)\left(1-\eta \alpha(b)\right)^{1-\sigma}\left(1+\frac{(1-\sigma)\alpha(b)}{2(1-\eta \alpha(b))\sqrt{\log_2 X}}\right).
\end{align*}
From \eqref{log expression} and \eqref{s_1/s_2}, we obtain
\begin{align*}
&\left(\tau_{\eta, X}+\frac{b\alpha^{\sigma}(b)(\log X)^{1-\sigma}}{2(1-\eta \alpha(b))^{\sigma}(\log_2 X)^{1/2+\sigma}}\right)\sideset{}{^\flat}\sum_{\substack{|d|\le X}}R_d^2\leq \bigg|\sideset{}{^\flat}\sum_{\substack{|d|\le X}}\log L(\sigma, \chi_d)R_d^2\bigg|\\
&\leq \bigg|\sideset{}{^\flat}\sum_{\substack{|d|\le X\\ \log |L(\sigma, \chi_d)|\leq \tau_{\eta, X}}}\log L(\sigma, \chi_d)R_d^2\bigg|+\bigg|\sideset{}{^\flat}\sum_{\substack{|d|\le X\\ \log |L(\sigma, \chi_d)|> \tau_{\eta, X}}}\log L(\sigma, \chi_d)R_d^2\bigg|\\
&\leq \tau_{\eta, X} \sideset{}{^\flat}\sum_{\substack{|d|\le X}}R_d^2+ \bigg|\sideset{}{^\flat}\sum_{\substack{|d|\le X\\ \log |L(\sigma, \chi_d)|> \tau_{\eta, X}}}\log L(\sigma, \chi_d)R_d^2\bigg|.
\end{align*}
From the computation of $S_2(\sigma)$ in the proof of Theorem \ref{main theorem 3} together with \eqref{resonator length for sigma} gives us
\begin{align*}
\bigg|\sideset{}{^\flat}\sum_{\substack{|d|\le X\\ \log |L(\sigma, \chi_d)|> \tau_{\eta, X}}}\log L(\sigma, \chi_d)R_d^2\bigg|&\geq \frac{b\alpha^{\sigma}(b)(\log X)^{1-\sigma}}{2(1-\eta \alpha(b))^{\sigma}(\log_2 X)^{1/2+\sigma}} \sideset{}{^\flat}\sum_{\substack{|d|\le X}}R_d^2\\
&\geq X^{\left(1+O(1/\log_2 X)\right)\left(1+\theta(b)\right)}.
\end{align*}
On the other hand, from \eqref{initial log expression}, we see that
\[
\max_{|d|\leq X}\log |L(\sigma, \chi_d)|\ll (\log X)^{\frac{8(1-\sigma)}{2\sigma-1}}/\log_2 X.
\]
This bound yields 
\begin{align*}
&\bigg|\sideset{}{^\flat}\sum_{\substack{|d|\le X\\ \log |L(\sigma, \chi_d)|> \tau_{\eta, X}}}\log L(\sigma, \chi_d)R_d^2\bigg|\leq \max_{|d|\leq X} R_d^2 \max_{|d|\leq X}\log |L(\sigma, \chi_d)|\frac{X}{\zeta(2)}\Phi_X(\sigma, \eta)\\
&\leq 2e^{\gamma}\frac{X}{\zeta(2)}\Phi_X(\sigma, \eta)  X^{-2a\log(1-b)(1+O(1/\log_2X))}(\log X)^{\frac{8(1-\sigma)}{2\sigma-1}}/\log_2 X,
\end{align*}
where we use that $
|R_d|\leq X^{-a\log(1-b)(1+O(1/\log_2X))}.
 $ 

Comparing the above lower and upper bound, and using the choice of $a$, we finally conclude 
\[
\Phi_X(\sigma, \eta)\geq X^{-\frac12 \left(1-\eta \alpha(b)\right)\left(1+O\left(1/\sqrt{\log_2 X}\right)\right)},
\]
which finishes the proof.

  \end{document}